\documentclass[11pt, oneside]{article}   

\usepackage{amsmath}
\usepackage{amsfonts}
\usepackage{amsthm}
\usepackage{amssymb}
\usepackage{mathrsfs}
\usepackage{cite}
\usepackage{tikz}
\usepackage{verbatim}
\usepackage{amsmath,amsfonts,amsthm,amssymb}

\newcommand{\vect}[1]{\overset{\rightharpoonup}{#1}}

\newif\ifshowvc
\showvctrue
\showvcfalse  
\makeatletter
\def\blfootnote{\xdef\@thefnmark{}\@footnotetext}
\makeatother
\ifshowvc
\input{vc}
\fi

\theoremstyle{plain}
\newtheorem{theorem}{Theorem}[section]

\newtheorem{corollary}[theorem]{Corollary}
\newtheorem{lemma}[theorem]{Lemma}

\theoremstyle{definition}
\newtheorem{definition}[theorem]{Definition}

\theoremstyle{remark}
\newtheorem{remark}[theorem]{Remark}

\newcommand{\beq}{\begin{align*}}
\newcommand{\eeq}{\end{align*}}

\newcommand{\ppp}{\[\begin{aligned}}
\newcommand{\ooo}{\end{aligned}\]}

\title{On Taylor series of zeros of complex-exponent polynomials}
\author{Mario DeFranco}
\begin{document}

\maketitle

\abstract{We prove a factorization formula for the Taylor series coefficients of a zero of a polynomial as a function of the polynomial's coefficients. This result extends to more general functions which we call ``complex-exponent polynomials". To prove this formula, we prove theorems about derivations on commutative rings. We also show that, when applied to polynomials, our formula recovers the results of Sturmfels obtained with GKZ systems (``Solving algebraic equations in terms of $\mathcal{A}$-hypergeometric series". Discrete Math. 210 (2000) pp. 171-181) }.

\section{Introduction} \label{Intro}

\subsection{Background}

Given a $d$-tuple $\vect{a}$ of complex numbers 
\[
\vect{a}=(a_1, \ldots, a_{d}),
\]
 a polynomial $p(z)$ is a function 
\begin{align*}
&p\colon  \mathbb{C} \rightarrow \mathbb{C}\\
&z  \mapsto \sum_{k=1}^{d} a_k z^{k-1}. 
\end{align*}
If $a_d\neq 0$, then $p(z)$ has degree $d-1$. A zero (or root) of $p(z)$ is a number $\phi \in \mathbb{C}$ such that 
\[
p(\phi)=0.
\]
As $\phi$ depends on the coefficients $a_1, \ldots, a_d$, we may think of $\phi$ as a function of $\vect{a}$
\[
\phi = \phi(\vect{a}).
\]

There has been much work to describe this dependence on $\vect{a}$. For degrees 1 through 4 there are the linear, quadratic, cubic, and quartic formulas, respectively, that give all zeros of $p(z)$ in $\mathbb{C}$. These are known as the solutions by radicals; that is, $\phi(\vect{a})$ is a function consisting of a finite number of applications of addition, subtraction, multiplication, division, and taking $n$-th roots (raising to the exponent $n^{-1}$), where $n$ is a positive integer. Knowledge of the quadratic formula stems from antiquity; for example, Babylonian cuneiform tablets from the second millennium B.C. describe methods to solve a quadratic equation (see Berriman \cite{Berriman}). The cubic formula was discovered in the sixteenth century A.D. (see Boyer and Merzbach \cite{Boyer} and Guilbeau \cite{Guilbeau}), due to Scipione Del Ferro and also being attributed to Niccol\`o Tartaglia and Gerolamo Cardano, whose book \textit{Ars Magna}  \cite{Cardano}, published in 1545, contains this formula. This book also contains the quartic formula discovered by Lodovico Ferrari (see O'Connor and Robertson \cite{O'Connor}). A solution by radicals does not exist for degree 5 or greater by the Abel-Ruffini theorem (\cite{Abel}, \cite{Ruffini}) proved completely by Niels Henrik Abel in 1824. Evariste Galois also proved this fact in 1831 (see \cite{Galois} and Radloff \cite{Radloff}). 

Nevertheless, the function $\phi(\vect{a})$ can still be described using infinite series for a polynomial of any degree. To study these series, Birkeland \cite{Birkeland} used Lagrange inversion; Mayr \cite{Mayr} used  a system of differential equations; and Herrera \cite{Herrera} used reversion of Taylor series. Sturmfels \cite{Sturmfels2000} considers the zero as a solution to a system of $\mathcal{A}$-hypergeometric differential equations introduced by Gel'fand, Kapranov, and Zelevinsky \cite{Gelfand89}, \cite{Gelfand90}. McDonald \cite{McDonald} and Passare and Tsikh \cite{Passare} also study these hypergeometric series.

In this paper we calculate the Taylor series coefficients of $\phi(\vect{a})$ directly and prove the main result (Theorem \ref{t main}) that they have a factorization formula. This theorem applies not only to polynomials but also to more general functions which we call ``complex-exponent polynomials" defined in section \ref{s statement}. We prove Theorem \ref{t main} in section \ref{s proof}; we use lemmas about Stirling numbers of the first and second kind from section \ref{s Stirling}, and also theorems involving sums over set partitions and subsets from section \ref{s set theorems}. In section \ref{s recovery}, we prove in Corollary \ref{c recovery} that this formula, when applied to polynomials, recovers the results of Sturmfels \cite{Sturmfels2000}. 

In section \ref{s set theorems}, the final step in the chain of reasoning is Theorem \ref{t deriv set} about an arbitrary derivation on a commutative ring. We prove another such theorem (Theorem \ref{t s tau}) in section \ref{s alt} which is used to give an alternative proof of Theorem \ref{t main} in the case $\beta=1$. 

\subsection{Statement of results}\label{s statement}

We define a complex-exponent polynomial using essentially the same definition above for a polynomial but allowing the exponents of $z$ to be complex numbers instead of natural numbers. This complex exponentiation is defined in the conventional way by changing the domain from $\mathbb{C}$ to the Riemann surface $L$ for the logarithm. This surface $L$ is parametrized by 
\[
L=\{(r,\theta,n)\colon r \in \mathbb{R}^+, \theta \in (-\pi, \pi], n \in \mathbb{Z}\}.
\]
Then for $z \in L$ corresponding to $(r,\theta,n)$, define $z^\gamma$ by 
\[
z^\gamma=e^{\gamma \ln(r)+i\gamma \theta+ 2 \pi i n \gamma} \in \mathbb{C}.
\]
\begin{definition}
Let $\vect{a}$ and $\vect{\gamma}$ be two $d$-tuples of complex numbers 
\begin{align*}
\vect{a} = (a_1, \ldots, a_d) \\ 
\vect{\gamma} = (\gamma_1, \ldots, \gamma_d).
\end{align*}
Define a complex-exponent polynomial $p(z;\vect{a},\vect{\gamma})$ to be a function of the form
\[
p: L \rightarrow \mathbb{C}
\]
\[
z \mapsto \sum_{k=1}^d a_k z^{\gamma_k}
\]
which we abbreviate as $p(z)$.
\end{definition}

Now we will consider a complex-exponent polynomial $f(z)$ which is a ``modification" of another complex-exponent polynomial $g(z)$. Specifically,
suppose we fix a $d$-tuple $\vect{\gamma}$ and also non-zero complex numbers $b$ and $\beta$. Then let $g(z)$ be the complex-exponent polynomial
\[
g(z) = 1 + b z^\beta.
\] 
We call $g(z)$ the ``base function".

\begin{definition}
Define $f(z; g, \vect{a}, \vect{\gamma})$ to be the complex-exponent polynomial 
\begin{align*}
z &\mapsto g(z)+\sum_{i=1}^d a_i z^{\gamma_i}
\end{align*}
which we abbreviate as $f(z)$.
\end{definition}

We next show how to express the Taylor series of a zero of $f(z)$ using a zero of $g(z)$. 

Let
\[
b= r_0 e^{i \theta_0}
\]
 for some $r_0>0$ and $\theta_0 \in (-\pi,\pi]$, and  
 \[
 \mathrm{Re}(\beta) = \beta_1 \,\text{ and }   \mathrm{Im}(\beta) = \beta_2.
 \]
Now for each $m \in \mathbb{Z}$, it is straightforward to check that $g(z)$ has a simple zero in $L$ corresponding to $(r,\theta,n)$
\begin{equation} \label{alpha formula}
(r,\theta,n)=(e^{\frac{\beta_2( (2m+1)\pi -\theta_0) - \beta_1 \ln(r_0)}{|\beta|^2}}, \frac{\beta_1((2m+1)\pi - \theta_0) +\beta_2 \ln(r_0)}{|\beta|^2} -2\pi n, n)
\end{equation}
where $n$ is chosen so that $\theta \in (-\pi, \pi]$. Let $\alpha$ denote one of these zeros. In what follows we will need not the formula \eqref{alpha formula} for $\alpha$, only the fact that $\alpha$ exists. 

We thus consider how this zero $\alpha$ varies as we vary $a_i$. That is, let $a_i$ be the coordinate variables of $\mathbb{C}^d$, and suppose there is a neighborhood $U$ of $\vect{0}$ in $\mathbb{C}^d$ and a function $\phi(\vect{a})$
\[
\phi \colon U \rightarrow L
\]
differentiable in the variables $a_1, \ldots, a_d$ at $\vect{0}$ for all orders such that
\begin{align}
f(\phi(\vect{a}); g, \vect{a}, \vect{\gamma})&=0 \text{ for all } \vect{a} \in U  \label{f=0}\\ 
\phi(\vect{0}) &= \alpha \label{base point}.
\end{align}
Then the above two equations are sufficient to calculate the Taylor series coefficients of $\phi(\vect{a})$ about $\vect{0}$.

Let $\vect{n}$ denote a $d$-tuple of non-negative integers 
\[
\vect{n} = (n_1, \ldots, n_d)
\]
and denote  
\[
\Sigma \vect{n} = \sum_{i=1}^d n_i.
\]
Let $\partial_{\vect{n}}$ denote the partial derivative operator
\[
\partial_{\vect{n}} = \prod_{i=1}^d (\frac{\partial}{\partial a_i})^{n_i}
\]
and for any function 
\[
\psi \colon : U \rightarrow \mathbb{C}
\]
denote
\[
\partial(\psi, I)=\partial_{\vect{n}} \psi(\vect{a}) |_{\vect{0}}.
\]

Our main result is the next theorem.
\begin{theorem} \label{t main}
With the above notation and when $\Sigma \vect{n} \geq 1$,
\[
\partial_{\vect{n}} \phi(\vect{a}) |_{\vect{0}} = -\frac{\alpha^{1 + \sum_{i=1}^d n_i(\gamma_i-1)}}{g'(\alpha)^{\Sigma \vect{n}}} \prod_{i=1}^{\Sigma \vect{n} -1} (-1+i \beta - \sum_{i=1}^d n_i\gamma_i)
\]
where 
\[
g'(\alpha) = b \beta \alpha^{\beta -1}.
\]
\end{theorem}

\section{Proof of Taylor Series Coefficient Formula}\label{s proof}

We re-express Theorem \ref{t main} as Theorem \ref{t main multiset} using different notation which will be used in its proof.
First we present notation for multisets and multiset partitions. For an integer $M\geq 0$, we let $[1,M]$ denote 
\[
[1,M] = \{i \in \mathbb{Z} \colon 1 \leq i \leq M\}.
\]

\begin{definition}
For a positive integer $N$, define an ordered multiset $I$ of $[1,d]$ to be an $N$-tuple of integers 
\[
I = (I(1), \ldots, I(N))
\]
with $1 \leq I(i) \leq d$. We say that the order $|I|$ is $N$. We define the multiplicity $\mathrm{multiplicity}(I, n)$ of $n$ in $I$ as the number of indices $i$ such that $I(i)=n$.
Let $\mathrm{Multiset}(d)$ denote the set of these ordered multisets. 

For a positive integer $k$, define a set partition $s$ of $[1,N]$ with $k$ parts to be a $k$-tuple
\[
s = (s_1, \ldots, s_k)
\]
where $s_i$ are pairwise disjoint non-empty subsets of $[1,N]$, 
\[
\bigcup_{i=1}^k s_i = [1,N],
\]
and 
\[
\min(s_i) < \min(s_j) \text{ for } i <j.
\]
We also write a set $s_i$ as an $m$-tuple
\[
s_i = (s_i(1), \ldots, s_i(m))
\]
where $m=|s_i|$ and 
\[
s_i(j) < s_i(l) \text{ for } j <l.
\]
Let $S(N,k)$ denote the set of such $s$. If $H$ is any finite set of integers, we similarly denote $S(H,k)$ to be the set of all set partitions of $H$ into $k$ non-empty parts.

For a multiset $I$ and a set partition $s \in S(|I|,k)$, define a multiset partition $J$ of $I$ with $k$ parts to be a $k$-tuple
 \[
 J = (J_1, \ldots, J_k)
 \]
 where $J \in \mathrm{Multiset}(d)$ is given by 
 \[
 J_i = (I(s_i(1)), \ldots, I(s_i(m)))
 \]
 where $m = |s_i|$. Thus the multiset partitions of $I$ with $k$ parts are in bijection with the set partitions in $S(|I|, k)$. 
Let $\mathrm{Parts}(I,k)$ denote the set of multiset partitions of $I$. 
 We let $ I(\hat{h})$ denote the ordered multiset obtained from $I$ by removing the element at the $h$-th index:
  \[
 I(\hat{h}) = (I(1), \ldots, I(h-1), I(h+1), \ldots,I(N)).
 \]
 We use the notation 
 \[
 \sum_{m \in I} \gamma_m = \sum_{i=1}^N \gamma_{I(i)}
 \]
where $N = |I|$.
\end{definition} 
We also use the falling factorial applied to indeterminates, where ``indeterminate" refers to an arbitrary element of some polynomial ring over $\mathbb{Z}$. 
\begin{definition}
For an integer $k \geq 0$ and an indeterminate $x$, define the falling factorial 
\[
(x)_k = \prod_{i=1}^k(x-i+1).
\]
\end{definition}

\begin{theorem} \label{t main multiset}
With the above notation and an ordered multiset $I \in  \mathrm{Multiset}(d)$ with $|I|\geq 1$,
\[
\partial(\phi, I) = -\frac{\alpha^{1 + \sum_{m\in I} (\gamma_{m}-1)}}{g'(\alpha)^{|I|} } (-\beta)^{|I|-1}( \beta^{-1}-1+\beta^{-1}\sum_{m \in I} \gamma_{m} )_{|I|-1}
\]
where 
\[
g'(\alpha) = b \beta \alpha^{\beta -1}.
\]
\end{theorem}
\begin{proof}
Recall by construction 
\begin{align*}
0 &= f(\phi(\vect{a})) \\ 
 &= g(\phi(\vect{a})) + \sum_{i=1}^d a_i\phi(\vect{a})^{\gamma_i}.
\end{align*}
We apply $\partial_I$ to both sides the equation 
and then set $\vect{a}=\vect{0}$.
If $I = (i)$, then we obtain 
\[
0=g'(\alpha)\partial(\phi,I)+ \alpha^{\gamma_i}.
\]
Solving for $\partial(\phi,I)$ yields 
\[
\partial(\phi,I) = \frac{-\alpha^{\gamma_i}}{g'(\alpha)}.
\]
This proves the theorem when $|I|=1$. 

Now we use induction on $|I|$. Given an $I \in  \mathrm{Multiset}(d)$ with $|I| \geq 2$, assume the theorem is true for all $I' \in  \mathrm{Multiset}(d)$ with $|I'| <|I|$.
Given any function $\psi(\vect{a})$ 
\[
\psi\colon U \rightarrow L,
\]
it follows from the definitions that
\begin{align}
\partial(f \circ \psi,I)  &= \sum_{h=1}^{|I|} \sum_{k=1}^{|I|-1} (\gamma_{I(h)})_k\psi(\vect{0})^{\gamma_{I(h)}-k}\sum_{J \in \mathrm{Parts}(I(\hat{h}), k)}\prod_{i=1}^k \partial(\psi,J_i)) \label{psi f} \\ 
&+  \sum_{k=1}^{|I|} g^{(k)}(\psi(\vect{0})) \sum_{J \in \mathrm{Parts}(I, k)}\prod_{i=1}^k \partial(\psi,J_i). \label{psi g}
\end{align}
For example, a term on the right side of line \eqref{psi f} corresponds to applying the partial derivative $\displaystyle \frac{\partial}{\partial a_{I(h)}}$ to the coefficient $a_{I(h)}$ of $f(z)$, and then applying $k$ other partial derivatives $\displaystyle \frac{\partial}{\partial a_{J_i(1)}}$ to the power of 
\[
\phi(\vect{a})^{\gamma_{I(h)}-i+1}
\]
proceeding from $i=1$ to $i=k$. Each of these $k$ applications by the chain rule results in a factor of 
\[
\frac{\partial \phi(\vect{a})}{\partial a_{J_i(1)}}.
\] 
Every other $m \in J_i$ then corresponds to applying $\displaystyle \frac{\partial}{\partial a_m}$ to this factor. Line \eqref{psi g} arises similarly. 

Now substitute $\phi(\vect{a})$ for $\psi(\vect{a})$. As above, since $f(\phi(\vect{a}))$ is identically zero by construction, so is $\partial(f\circ \phi, I)$ for any $I$. We obtain
\begin{align}
0  &= \sum_{h=1}^{|I|} \sum_{k=1}^{|I|-1} (\gamma_{I(h)})_k\alpha^{\gamma_{I(h)}-k}\sum_{J \in \mathrm{Parts}(I(\hat{h}), k) }\prod_{i=1}^k \partial(\phi,J_i) \label{phi f}\\ 
&+  \sum_{k=2}^{|I|} g^{(k)}(\alpha) \sum_{J \in \mathrm{Parts}(I, k)}\prod_{i=1}^k \partial(\phi,J_i)\label{phi g}\\ 
&+g'(\alpha) \partial(\phi,I) \label{phi I}.
\end{align}

By the induction hypothesis, the right of line \eqref{phi f} becomes 
\begin{align*}
\sum_{h=1}^{|I|} \sum_{k=1}^{|I|-1} (\gamma_{I(h)})_k\alpha^{\gamma_{I(h)}-k}\sum_{J \in \mathrm{Parts}(I(\hat{h}), k)}\prod_{i=1}^k  \frac{-\alpha^{1+\sum_{m \in J_i} (\gamma_m - 1)}}{g'(\alpha)^{|J_i|}}  (-\beta)^{|J_i|-1}(\beta^{-1}-1+\beta^{-1}\sum_{m \in J_i}\gamma_m)_{|J_i|-1}\\ 
\end{align*}
which we simplify to 
\begin{equation}
\frac{\alpha^{1+\sum_{m \in I} (\gamma_m-1)}(-\beta)^{|I|-2} }{g'(\alpha)^{|I|-1}}  \sum_{h=1}^{|I|} \sum_{k=1}^{|I|-1} -(\gamma_{I(h)})_k  (\beta^{-1})^{k-1}\sum_{J \in \mathrm{Parts}(I(\hat{h}), k)}\prod_{i=1}^k (\beta^{-1}-1+\beta^{-1}\sum_{m \in J_i}\gamma_m)_{|J_i|-1}. \label{simplify to}\\ 
\end{equation}
We apply Theorem \ref{t nu} to the quantity 
\begin{equation} \label{t main multiset quantity}
(\beta^{-1})^{k-1}\sum_{ J \in \mathrm{Parts}(I(\hat{h}), k)}\prod_{i=1}^k (\beta^{-1}-1+\beta^{-1}\sum_{m \in J_i}\gamma_m)_{|J_i|-1}
\end{equation}
with $\nu=\beta^{-1}$; $N=|I(\hat{h})|$; and $x_i=\beta^{-1}\gamma_{I(\hat{h})(i)}$ 
to see that the quantity \eqref{t main multiset quantity} is equal to $c_k$, where 
\[
c_k =\frac{1}{(k-1)!}\sum_{r=0}^{k-1}(-1)^{k-1-r} {k-1 \choose r} (\beta^{-1}(r+1) -1+\beta^{-1}\sum_{m\in I(\hat{h})} \gamma_m)_{|I|-2}.
\]
Now the sum 
\[
-\sum_{k=1}^{|I|-1}(\gamma_{I(h)})_kc_k 
\]
is equal to 
\begin{equation}\label{sum gamma_Ih}
-\gamma_{I(h)} \sum_{k=1}^{|I|-1}c_k (\gamma_{I(h)}-1)_{k-1}.
\end{equation}
Noting that $c_k$ are the form of coefficients in a Newton series in $\gamma_{I(h)}$, we apply Lemma \ref{Newton series} with
\[
F(x) = (\beta^{-1}x -1+\beta^{-1}\sum_{m\in I(\hat{h})} \gamma_m)_{|I|-2}
\]
and set $x$ to be $\gamma_{I(h)}$ to obtain that the expression \eqref{sum gamma_Ih} is equal to 
\[
-\gamma_{I(h)}(\beta^{-1}\sum_{m \in I} \gamma_m - 1)_{|I|-2}. 
\]
Summing over $h$, we get that \eqref{simplify to} is equal to 
\begin{equation}\label{phi f final}
\frac{\alpha^{1-|I|+\sum_{m \in I} \gamma_m}(-\beta)^{|I|-1} }{g'(\alpha)^{|I|-1}}(\beta^{-1}\sum_{m \in I}\gamma_m)_{|I|-1}. 
\end{equation}

Now consider the terms at line \eqref{phi g}. Applying
\[
g^{(k)}(\alpha) = b (\beta)_k \alpha^{\beta -k},
\]
we use the induction hypothesis and proceed as done for line \eqref{phi f} to see that the sum of these terms is equal to 
\[
\frac{-b \beta \alpha^\beta (-\beta)^{|I|-1} \alpha^{\sum_{m \in I} (\gamma_m-1)}}{g'(\alpha)^{|I|}} \sum_{k=2}^{|I|} \frac{(\beta-1)_{k-1}}{(k-1)!}\sum_{r=0}^{k-1} (-1)^{k-1-r}{k-1 \choose r}((r+1)\beta^{-1}-1+\beta^{-1} \sum_{m \in I}\gamma_m).
\]
We add and subtract the term corresponding to $k=1$; this term is 
\[
(\beta^{-1}-1+\beta^{-1} \sum_{m \in I}\gamma_m)_{|I|-1}.
\]
The sum including $k=1$ is now a Newton series in $\beta$. Using Lemma \ref{Newton series} it is equal to 
\[
(\beta^{-1} \sum_{m \in I}\gamma_m)_{|I|-1}.
\]
Therefore line \eqref{phi g} is equal to 
\begin{align*}
&\frac{-b \beta \alpha^\beta (-\beta)^{|I|-1} \alpha^{\sum_{m \in I} (\gamma_m-1)}}{g'(\alpha)^{|I|}} \left(  (\beta^{-1} \sum_{m \in I}\gamma_m)_{|I|-1} - (\beta^{-1}-1+\beta^{-1} \sum_{m \in I}\gamma_m)_{|I|-1}\right).
\end{align*}
Combining with result \eqref{phi f final} and simplifying using 
\[
g'(\alpha) = b \beta \alpha^{\beta-1},
\]
yields the equation 
\[
0 = \frac{\alpha^{1+\sum_{m \in I} (\gamma_m-1)} (-\beta)^{|I|-1} }{g'(\alpha)^{|I|-1}} (\beta^{-1}-1+\beta^{-1} \sum_{m \in I}\gamma_m)_{|I|-1}+ g'(\alpha) \partial(\phi,I).
\]
Solving for $\partial(\phi,I)$ completes the proof.
\end{proof}

\begin{lemma} \label{Newton series}
Suppose $F(x)$ is a polynomial of degree $m$. Then
\[
F(x) = \sum_{k=1}^{m+1} \frac{(x-1)_{k-1}}{(k-1)!}\sum_{r=0}^{k-1} (-1)^{k-1-r} {k-1 \choose r}F(r+1).
\]
\end{lemma} 
\begin{proof}
The Newton series of a polynomial $P(x)$ of degree $m$ is 
\[
P(x) = \sum_{k=0}^m {x \choose k} \sum_{r=0}^{k}(-1)^{k-r} {k \choose r}P(r). 
\]
We prove this standard formula now. Both sides are polynomials in $x$ of degree $m$. Evaluating $x$ at an integer $c, 0 \leq c \leq m$ on the right side and collecting the terms $P(r)$ for a fixed $r$  give 
\[
P(r)\sum_{k=r}^c (-1)^{k-r}{c \choose k} {k \choose r}. 
\] 
Applying the identity 
\[
{c \choose k} {k \choose r} = {c -r \choose k-r} {c \choose r}
\]
gives 
\[
P(r) {c \choose r} \sum_{k=r}^c (-1)^{k-r}{c -r \choose k-r}. 
\]
which is equal to $0$ if $c\neq r$ and $P(c)$ is $c=r$. Thus both sides are equal at $m+1$ distinct inputs, and thus are equal as polynomials. This proves the Newton series formula.

To prove the lemma, we use $F(x+1)$ for $P(x)$ and use 
\[
{x \choose k} = \frac{(x)_k}{k!},
\]
then substitute $x \mapsto x-1$ and re-index $k \mapsto k-1$. This completes the proof. 
\end{proof}

\section{Stirling number results}\label{s Stirling}
Before proving Theorem \ref{t nu} we present notation for Stirling numbers. 
\begin{definition}
For integers $N,r\geq 0$, define the unsigned Stirling number of the first kind 
\[
{N\brack r }
\]
to be $(-1)^{N-r}$ times the coefficient of $Y^r$ when $(Y)_N$ is expressed in the monomial basis of polynomials in the indeterminate $Y$.
Define the Stirling number of the second kind 
\[
{N\brace r }
\]
to be the coefficient of $(Y)_r$ when $Y^N$ is expressed in the falling factorial basis of polynomials. 

Equivalently we may define the Stirling numbers by the recursive relations for $N \geq r$ by
\begin{equation} \label{Stirling 1 rec} 
{N \brack r}+ N{N \brack r+1} = {N+1 \brack r+1}
\end{equation}

\begin{equation} \label{Stirling 2 rec} 
{N \brace r}+ (r+1){N \brace r+1} = {N+1 \brace r+1}
\end{equation}
and the conditions
\begin{align*}
{N \brace 1} &= {N \brack N}=1 \text{ for } N \geq 1 \\ 
{0\brack 0}&=1\\
{N \brace r} &= {N \brack r}=0 \text{ for } N<r. 
\end{align*}
It is straightforward to show that these two definitions are equivalent.
\end{definition}

\begin{definition}
For integers $N,r\geq0$ and an indeterminate $X$, define 
\[
{N \brack r}_X
\]
to be the coefficient of $Y^r$ in 
\[
(Y+X-1)_N.
\]
If $N<0$ and $r \geq 0$, then let $\displaystyle {N \brack r}_X$ denote 0. 
\end{definition}

\begin{remark}
It follows from the definitions that 
\begin{align}
{N \brack r}_0 &=  (-1)^{N-r}{N+1 \brack r+1} \label{nu Stirling 0}\\ 
{N \brack r}_1 &=(-1)^{N-r} {N \brack r} \label{nu Stirling 1}
\end{align}
\end{remark}

Next we prove lemmas used in Section \ref{s set theorems}.
\begin{lemma} \label{l Stirling 2 a n}
For integers $a,n \geq 0$, 
\begin{equation} \label{Stirling 2 a n}
\frac{1}{a!}\sum_{r=0}^{a}(-1)^{a-r}{a \choose r}(r+1)^{n}  = {n+1\brace a+1}.
\end{equation}
\end{lemma} 
\begin{proof}
Denote the left side of equation \eqref{Stirling 2 a n} by $F(n+1,a+1)$. First, we claim that 
\[
(a+1)F(n+1,a+1)+ F(n+1,a) =F(n+2,a+1). 
\]
Given an $r$ with $0 \leq r \leq a$ and taking the coefficients of $(r+1)^n$ in the above equation, we see that the claim is implied by the identity 
\[
(a+1){a \choose r} - a {a+1 \choose r} = {a \choose r}(r+1).
\] 


Second, we claim that $F(n+1,a+1)=0$ when $n<a$. We have that 
\begin{equation} \label{F d dt}
F(n+1,a+1) = \frac{1}{a!}(\frac{d}{dt} t)^n (t-1)^a |_{t=1}
\end{equation}
using the binomial expansion of $(t-1)^a$. Now evaluate the the right side of equation \eqref{F d dt} by applying $\displaystyle \frac{d}{dt}$ to products of $t$ and $(t-1)$; every term has a factor of $(t-1)$ when $n<a$. This proves the second claim.

Third, we have for $n \geq 0$
\[
F(n+1,1)=1.
\]
Thus $F(n+1,a+1)$ satisfies the recursive definition and initial condition of $\displaystyle {n+1\brace a+1}$ when $0 \leq a \leq n$. This completes the proof.
\end{proof}


\begin{lemma} \label{Stirling X Y}
For integers $N\geq r \geq0$ and indeterminates $X$ and $Y$,
 \[
{N \brack r}_{X+Y} = \sum_{i=0}^{N-r}X^i {r+i \choose i}  {N \brack r+i}_Y
\]
\end{lemma}
\begin{proof} 
By definition $\displaystyle {N \brack r}_{X+Y}$ is equal to
\[
\sum_{w\subset[1,N], |w|=N-r } \prod_{i \in w} (X+Y-i).
\]
In a term corresponding to a subset $w$ with order $N-r$, the coefficient of $X^i$ is 
\[
\sum_{w' \subset w, |w'| = N-r-i} \prod_{j \in w'}(Y-j).
\]
Given any subset $v \subset [1,N]$ of order $N-r-i$, there are $\displaystyle {r+i \choose i}$ subsets of $[1,N]$ of order $N-r$ that contain $v$. Therefore 
\begin{align*}
{N \brack r}_{X+Y}  &= \sum_{i=1}^{N-r} X^i  {r+i \choose i} \sum_{v \subset [1,N], |v| = N-r-i} \prod_{j\in v}(Y-j) \\ 
&= \sum_{i=1}^{N-r} X^i  {r+i \choose i} {N \brack r+i}_Y.
\end{align*}
This completes the proof.
\end{proof}
\begin{lemma} \label{l Stirling gf}For an integer $n \geq 0$, as formal power series
\begin{equation}\label{Stirling gf}
\frac{(-\ln(1-t))^n}{n!} = \sum_{k=0}^\infty \frac{{k \brack n}}{k!}t^k. 
\end{equation}
\end{lemma}
\begin{proof}
We use induction on $n$. The lemma is true when $n=0$, for then both sides are equal to 1. Assume it is true for some $n \geq 0$. 
Multiply both sides of equation \eqref{Stirling gf} by 
\[
\frac{1}{1-t} = \sum_{k=0}^\infty t^k
\]
and integrate. We thus have 
\begin{equation}\label{Stirling gf int}
\frac{(-\ln(1-t))^{n+1}}{(n+1)!} = \sum_{m=0}^\infty \frac{t^{m+1}}{m+1} \sum_{k=0}^m \frac{{k \brack n}}{k!}. 
\end{equation}
We claim 
\[
\sum_{k=0}^m \frac{{k \brack n}}{k!} = \frac{{m+1 \brack n+1}}{m!}.
\]
From the definition of the Stirling numbers of the first kind 
\[
\frac{{k \brack n}}{k!}= \frac{{k +1\brack n+1}}{k!}-\frac{{k \brack n+1}}{(k-1)!}.
\]
Summing both sides from $k=0$ to $k=m$ and using the fact that 
\[
{0\brack n+1}=0
\]
proves the claim and the induction step. This completes the proof.
\end{proof}

\begin{lemma} \label{l Stirling 2 sum}
For integers $0 \leq r \leq k$,
\begin{equation} \label{Stirling 2 sum}
\sum_{i=r}^{k} { i \brace r}{k \choose i}= {k+1\brace r+1}.
\end{equation}
\end{lemma}
\begin{proof}
 Let $F(k+1,r+1)$ denote the left side of equation \eqref{Stirling 2 sum}. 
 For a fixed $r\geq 0$, we use induction on $k$. The lemma is true when $k=r$. Assume the lemma is true for some $k \geq r$. Apply the identity 
\[
{k +1 \choose i } = {k \choose i }+{k \choose i-1 }
\]
to sum in $F(k+2,r+1)$ and re-arrange to obtain 
\begin{equation} \label{Stirling 2 sum re-arranged}
F(k+2,r+1)={r-1 \brace r-1 } {k \choose r-1 }+ \sum_{i=r}^{k+1}\left( {i \brace r }+{i+1 \brace r }\right) {k \choose i }.
\end{equation}

Now multiply equation \eqref{Stirling 2 sum} by $(r+1)$ and subtract the result from equation \eqref{Stirling 2 sum re-arranged} to obtain 
\begin{equation} \label{Stirling 2 result}
{r-1 \brace r-1 } {k \choose r-1 }+ \sum_{i=r}^{k+1}\left( {i \brace r }+{i+1 \brace r } - (r+1){i \brace r}\right) {k \choose i }. 
\end{equation}
From equation \eqref{Stirling 2 rec} the coefficient of $\displaystyle {k \choose i}$ is $\displaystyle {i \brace r-1}$, and by the induction hypothesis again,
equation \eqref{Stirling 2 result} is equal to $\displaystyle {k +1\brace r}$. 
We have thus shown that 
\[
F(k+2,r+1) - (r+1){k+1 \brace r+1} =  {k +1\brace r}
\]
and combining with equation \eqref{Stirling 2 rec} we have  
\[
F(k+2,r+1) = {k+2 \brace r+1}.
\]
This completes the proof.
\end{proof}

\section{Theorems about set partitions and subsets} \label{s set theorems}
In this section we prove Theorem \ref{t nu} used in the proof of Theorem \ref{t main multiset}, and also Theorems \ref{t Stirling set} and Theorems \ref{t deriv set} which are used to prove Theorem \ref{t nu}.

\begin{theorem} \label{t nu}
For integers $1 \leq k \leq N$, an indeterminate $\nu$, and $N$ indeterminates $x_i$, $1 \leq i \leq N$,
\begin{align}
&\frac{1}{(k-1)!}\sum_{r=0}^{k-1} (-1)^{k-1-r}{k-1 \choose r} ((r+1)\nu-1+\sum_{i=1}^N x_i)_{N-1} \label{nu 1}\\
=&\nu^{k-1}\sum_{s \in S(N,k)} \prod_{i=1}^{k} (\nu-1+\sum_{m\in s_i} x_m)_{|s_i|-1}. \label{nu 2}
\end{align}
\end{theorem}
\begin{proof}
We use induction on $k$. The theorem is true when $k=1$ for then both sides are equal to 
\[
(\nu-1+\sum_{i=1}^N x_i)_{N-1}.
\]
Assume that the theorem is true for all values less than some $k \geq 2$. 
We will expand both lines \eqref{nu 1} and \eqref{nu 2} into powers of $\nu$ and $x_N$, and show that the coefficients are equal.

First we expand line \eqref{nu 1} as 
\[
\sum_{n=0}^{N-1} \nu^{n}{N-1 \brack n}_{\sum_{i=1}^N x_i} \frac{1}{(k-1)!}\sum_{r=0}^{k-1}(-1)^{k-1-r}{k-1 \choose r} (r+1)^n
\]
which by Lemma \ref{l Stirling 2 a n} we may write as 
\begin{equation} \label{nu power left}
\sum_{n=k-1}^{N-1} \nu^{n}{N-1 \brack n}_{\sum_{i=1}^N x_i} {n+1\brace k}.
\end{equation}
Next, by the induction hypothesis, line \eqref{nu 2} is equal to
\[
\nu \sum_{w \subset[1,N], N\in w} 
\left(\sum_{n=k-2}^{N-|w|-1} \nu^{n}{n+1\brace k-1} {N-|w|-1 \brack n}_{\sum_{i \in w^c} x_i}\right)(\nu-1+\sum_{i \in w}x_i)_{|w|-1}.
\]
We expand the second factor of this sum in powers of $\nu$ to obtain
\[
\nu \sum_{w \subset[1,N], N\in w} 
\left(\sum_{n=k-2}^{N-|w|-1} \nu^{n}{n+1\brace k-1} {N-|w|-1 \brack n}_{\sum_{i \in w^c} x_i}\right)\left(\sum_{n=0}^{|w|-1} \nu^{n}{|w|-1 \brack n}_{\sum_{i \in w} x_i}\right).
\]
In the above sum, the coefficient of $\nu^{n_0}$ is
\begin{align} \label{nu power right}
\sum_{j=k-2}^{n_0-1}{j+1\brace k-1} \sum_{w \subset [1,N], N \in w}   {N-|w|-1 \brack j}_{\sum_{i \in w^c} x_i} {|w|-1 \brack n_0-1-j}_{\sum_{i \in w} x_i}
\end{align}
Now we consider the coefficient of $x_N^{m_0} \nu^{n_0}$. In a term of the inner sum above, consider the right factor. 
By Lemma \ref{Stirling X Y}, we have 
\[
{|w|-1 \brack n_0-1-j}_{\sum_{i \in w} x_i} = \sum_{m_0=0}^{|w|-n_0+j} x_N^{m_0} {n_0-1-j+m_0 \choose m_0}  {|w|-1 \brack n_0-1-j+m_0}_{\sum_{i \in w, i\neq N} x_i}. 
\]
Therefore the coefficient of $x_N^{m_0}$ in expression \eqref{nu power right} is 
\begin{equation} \label{x nu power right}
\sum_{j=k-2}^{n_0-1} {j+1\brace k-1}  {n_0-1-j+m_0 \choose m_0}     \sum_{v \subset [1,N-1]}   {|v^c|-1 \brack j}_{\sum_{i \in v^c} x_i}{|v| \brack n_0-1-j+m_0}_{\sum_{i \in v} x_i}
\end{equation}
where we have reindexed using the set $v = w \setminus \{N\}$, and $v^c$ denotes the complement of $v$ in $[1,N-1]$. We must prove that the above expression is equal to the coefficient of $x_N^{m_0} \nu^{n_0}$ in expression \eqref{nu power left}, which by Lemma \eqref{Stirling X Y} is
\begin{equation} \label{x nu power left}
 {n_0+1\brace k} {N-1 \brack n_0+m_0}_{\sum_{i=1}^{N-1} x_i} {n_0+m_0 \choose m_0}.
\end{equation}
Applying Theorem \ref{t Stirling set} to the inner sum of expression \eqref{x nu power right} with $a-1=j, b=n_0-1-j+m_0$, and $N-1=M$ yields 
\begin{equation} \label{x nu power right yield}
\sum_{j=k-2}^{n_0-1} {j+1\brace k-1}  {n_0-1-j+m_0 \choose m_0}    {n_0+m_0 \choose j+1}{N-1 \brack n_0+m_0}_{\sum_{i =1}^{N-1} x_i}.
\end{equation}
Equating expressions \eqref{x nu power left} and \eqref{x nu power right yield} and then simplifying the binomial coefficients shows that it is sufficient to prove the equation
\begin{align*}
{n_0+1 \brace k}=\sum_{j=k-2}^{n_0-1}{j+1\brace k-1}{n_0 \choose j +1}.\\
\end{align*}
This follows from Lemma \ref{Stirling 2 sum}. This complete the proof.
\end{proof} 
 
\begin{theorem} \label{t Stirling set} For integers $M\geq 0; a\geq 1$, and $b \geq 0$, and indeterminates $x_i, 1 \leq i \leq M$,
\begin{align}\label{Stirling set}
\sum_{w \subset [1,M]}  {|w^c|-1 \brack a-1}_{\sum_{i \in w^c} x_i} {|w| \brack b}_{\sum_{i \in w} x_i}= {a+b \choose a }{M \brack a+b}_{\sum_{i =1}^{M} x_i}
\end{align}
where $w^c$ denotes the complement of $w$ in $[1,M]$.
\end{theorem}
\begin{proof}
Fix an integer $l \geq 0$ and integers $n_i \geq 0$ for $1 \leq i \leq l$. Consider a term of the form 
 \begin{equation}\label{Stirling set term}
 \prod_{i=1}^l x_i^{n_i}.
 \end{equation}
 The coefficient of this term on the left side of equation \eqref{Stirling set} is 
 \begin{align} \label{Stirling set left}
 &\sum_{w \subset [1,M], v \subset [1,l], v \subset w, v^c \subset w^c}  \left( \frac{(\sum_{i\in v^c} n_i)!}{\prod_{i\in v^c} i!}\right) {a-1+\sum_{i\in v^c}n_i \choose a-1} (-1)^{|w^c|-a}{|w^c| \brack a+\sum_{i\in v^c}n_i}  \\
  &\times \left(\frac{(\sum_{i\in v} n_i)!}{\prod_{i\in v} i!}\right) {b+\sum_{i\in v}n_i \choose b}(-1)^{|w|-b}{|w|+1 \brack b+1+\sum_{i\in v}n_i}  \nonumber
 \end{align}
 where we have used Lemma \ref{Stirling X Y} and equation \eqref{nu Stirling 0}, and where $v^c$ denotes the complement of $v$ in $[1,l]$. The number of sets $w$ of order $j$ containing such $v$ is
 \[
 {M-l \choose j-|v|}
 \] 
 so we rewrite \eqref{Stirling set left} as 
 \begin{align} \label{Stirling set left rewrite}
 &\sum_{v \subset [1,l]}  \sum_{j=0}^M \left( \frac{(\sum_{i\in v^c} n_i)!}{\prod_{i\in v^c} i!}\right) {a-1+\sum_{i\in v^c}n_i \choose a-1} (-1)^{M-j-a}{M-j \brack a+\sum_{i\in v^c}n_i}  {M-l \choose j-|v|} \\
  &\times \left(\frac{(\sum_{i\in v} n_i)!}{\prod_{i\in v} i!}\right) {b+\sum_{i\in v}n_i \choose b}(-1)^{j-b}{j+1 \brack b+1+\sum_{i\in v}n_i}  \nonumber
 \end{align}
 
 Likewise the coefficient of the term \eqref{Stirling set term} on the right side of equation \eqref{Stirling set} is 
  \begin{equation} \label{Stirling set right}
 {a+b \choose a } \left( \frac{(\sum_{i=1}^l n_i)!}{\prod_{i=1}^l i!}\right) {a+b+\sum_{i=1}^l n_i \choose a+b}(-1)^{M-a-b}{M+1 \brack a+b+1+\sum_{i=1}^l n_i}.  
 \end{equation}
 Now equate expression \eqref{Stirling set left} and \eqref{Stirling set right} and simplify to obtain 
\begin{align} \label{Stirling set equate}
 &\sum_{v \subset [1,l]}  \sum_{j=0}^M a (a-1+\sum_{i\in v^c}n_i)! \frac{{M-j \brack a+\sum_{i\in v^c}n_i}}{((M-j) - |v^c|)!} (b+\sum_{i\in v}n_i)!\frac{{j+1 \brack b+1+\sum_{i\in v}n_i}}{(j-|v|)!}   \\ 
  &=  (a+b+\sum_{i=1}^l n_i)!\frac{{M+1 \brack a+b+1+\sum_{i=1}^l n_i}}{(M-l)!}.
   \end{align}
 Now multiply both sides of the above equation by $t^{M-l}$ and sum over $M\geq 0$ to yield 
\begin{align*} 
 &\sum_{v \subset [1,l]}  a (\frac{d}{dt})^{|v^c|}\left(\frac{(-\ln(1-t))^{a+\sum_{i\in v^c}n_i}}{a+\sum_{i\in v^c}n_i} \right) (\frac{d}{dt})^{|v|+1} \left(\frac{(-\ln(1-t))^{b+1+\sum_{i\in v}n_i}}{b+1+\sum_{i\in v}n_i}  \right)  \\ 
  &=  (\frac{d}{dt})^{l+1}\left(\frac{(-\ln(1-t))^{a+b+1+\sum_{i=1}^l n_i}}{a+b+1+\sum_{i=1}^l n_i} \right)
   \end{align*}
   where we have used Lemma \ref{l Stirling gf}. After applying one derivative from each power of $\frac{d}{dt}$, we see that the above equation is equivalent to 
\begin{align*}
 &\sum_{v \subset [1,l]}   (\frac{d}{dt})^{|v^c|-1}\left(\frac{a(-\ln(1-t))^{a-1+\sum_{i\in v^c}n_i} }{1-t}\right) (\frac{d}{dt})^{|v|} \left(\frac{(-\ln(1-t))^{b+\sum_{i\in v}n_i}}{1-t}  \right)  \\ 
  &=  (\frac{d}{dt})^{l}\left(\frac{(-\ln(1-t))^{a+b+\sum_{i=1}^l n_i}}{1-t} \right)
   \end{align*} 
   where in the case $v^c$ is empty we denote 
   \[
    (\frac{d}{dt})^{-1}\left(\frac{a(-\ln(1-t))^{a-1} }{1-t}\right) = (-\log(1-t))^a.
   \]
   This equation follows from Theorem \ref{t deriv set} with $R$ the ring of power series in $t$, $\delta$ differentiation with respect to $t$, and 
   \begin{align*}
   f_A = (-\log(1-t))^a\\ 
   f_B =  \frac{(-\log(1-t))^b }{1-t}\\ 
   f_i = (-\log(1-t))^{n_i}.
   \end{align*}
   This completes the proof.
\end{proof}

 \begin{theorem} \label{t deriv set}
 Let $R$ be a commutative ring and let $\delta \colon R \rightarrow R$ be a derivation. 
 For an integer $M\geq 0$ and elements $f_A, f_B, f_i \in R, 1 \leq i \leq M$, 
 \begin{equation}\label{deriv set}
\sum_{ w \subset [1,M] }\delta^{|w^c|-1} (f_A^{(1)} \prod_{i \in w^c} f_i ) \delta^{|w|} (f_B \prod_{i \in w} f_i )  = \delta^M (f_A f_B \prod_{i=1}^M f_i)
  \end{equation}
  where $w^c$ denotes the complement of $w$ in $[1,M]$; $\displaystyle   f_A^{(1)}$ denotes $\delta f_A$; and in the case $w^c$ is empty, $\displaystyle   \delta^{-1} f_A^{(1)}$ denotes $f_A$.
 \end{theorem} 
\begin{proof} Note that we do not require $R$ to contain 1. For an element $f \in R$ and integer $n\geq 0$, we denote
\[
n f = \sum_{i=1}^n f
\]
and say that $n$ is the coefficient of $f$. We also denote 
 \[
 f^{(n)} = \delta^n f.
 \]
 
 We use induction on $M$. The theorem is true when $M=0$, for then $w$ is empty and we have 
 \[
\delta^{-1} (f_A^{(1)}) \delta^0 f_B = f_A f_B.
 \]
 For an $M \geq 1$, assume the theorem is true for all values less than $M$. Consider a term of the form 
 \[
 f_A^{(n_A)} f_B^{(n_B)}\prod_{i=1}^M f_i^{(n_i)}.
 \]
In order for this term to arise on the left or right side of equation \eqref{deriv set}, we must have 
\begin{equation} \label{deriv set m}
M = n_A+n_B+\sum_{i=1}^M n_i.
\end{equation}
Suppose $n_i\geq 0$ for each $i, 1 \leq i \leq M$. Then $n_A=n_B=0$ and $n_i=1$ for each $i$. It can arise from the left side only when $w = [1,M]$, and the coefficient of this term on both sides is $M!$. 
 
 Thus suppose at least one of the $n_i=0$. By symmetry, we may consider a term of the form 
 \begin{equation} \label{deriv set l term}
 f_A^{(n_A)} f_B^{(n_B)}\prod_{i=1}^l f_i^{(n_i)}
 \end{equation}
with $0\leq l < M$; $n_A,n_B\geq 0$ and $n_i \geq 1$. The coefficient of this term on the left side of equation \eqref{deriv set} is 
\begin{equation} \label{deriv set coeff left}
\sum_{v \subset [1,l]} \left(\frac{(n_A-1 + \sum_{i \in v^c} n_i)!}{(n_A-1)!\prod_{i \in v^c} (n_i)!} \right) \left(\frac{(n_B + \sum_{i \in v} n_i)!}{(n_B)!\prod_{i \in v} (n_i)!} \right) {M-l \choose n_B  - |v|+\sum_{i\in v} n_i } 
\end{equation}
 where $v^c$ denotes the complement of $v$ in $[1,l]$. Here $v \subset w$ and 
\begin{align}
n_A-1+ \sum_{i \in v^c}  {n_i} = M-|w|-1 \label{wc order} \\
 n_B+ \sum_{i \in v}  {n_i} = |w| .\label{w order} 
\end{align}
 In a term of the above sum, the first two factors are the coefficients of 
 \[
 f_A^{(n_A)} \prod_{i \in v^c} f_i^{(n_i)} \text{ and }  f_B^{(n_B)} \prod_{i \in v} f_i^{(n_i)} 
 \]
 in 
 \[
 \delta^{n_A-1+ \sum_{i \in v^c}  n_i }(f_A^{(1)}\prod_{i \in v^c} f_i^{(n_i)}) \text{ and }  \delta^{n_B+ \sum_{i \in v}  n_i }(f_B\prod_{i \in v} f_i^{(n_i)}) 
 \]
 respectively. The third factor is the number of sets $w \subset [1,M]$ that contain $v$ and do not contain $v^c$. This number is 
 \[
 {M-l\choose |w|-|v|}
 \]
and then apply line \eqref{w order}. The coefficient of term \eqref{deriv set l term} on the right side of equation \eqref{deriv set} is 
\begin{equation} \label{deriv set coeff right}
\frac{(n_A+n_B+\sum_{i=1}^l n_i)!}{n_A!n_B!\prod_{i=1}^l (n_i)!}.
\end{equation}
Equating expressions \eqref{deriv set coeff left} and \eqref{deriv set coeff right} and simplifying with equation \eqref{deriv set m}, we obtain the equation 
\begin{align*}
&\sum_{v \subset [1,l]} n_A (n_A-1+\sum_{i\in v^c} n_i)_{|v^c|-1} (n_B+\sum_{i\in v} n_i)_{|v|}   \\ 
&= (n_A+n_B+\sum_{i=1}^l n_i)_{l}
\end{align*}
where in the case $v^c$ is empty we denote
\[
n_A (n_A-1)_{-1} =1.
\]
But because $l<M$, this equation follows from the induction hypothesis with $R$ being the ring of polynomials in the variable $t$, $\delta$ being differentiation with respect to $t$, and 
$f_A = t^{n_A}, f_B = t^{n_B}$, and $f_i = t^{n_i}$ and then setting $t=1$. This completes the proof.
\end{proof}

\section{Recovery of $\mathcal{A}$-hypergeometric series}\label{s recovery}
We now prove that Theorem \ref{t main multiset} recovers the formula of Sturmfels (\cite{Sturmfels2000}, Section 3). We follow the notation of that paper (including Definition \ref{Sturmfels 3} here) with slight alterations. Note in this section that $a_i$ and their fractional powers are treated as indeterminates as done in that paper. We prove Lemmas \ref{l vi12}, \ref{l X power series} and Theorem \ref{t recovery} to prove the the recovery in Corollary \ref{c recovery}.

Fix an integer $n \geq 2$ and two integers $0 \leq i_1 < i_2 \leq n$. Set $d=i_2 - i_1$. Let $\mathcal{A}$ denote the $2\times(n+1)$ matrix
\begin{equation*}
\mathcal{A}= 
\begin{pmatrix}
0 & 1 & \cdots & n \\
1 & 1 & \cdots & 1 \\
\end{pmatrix}.
\end{equation*}
Viewing $\mathcal{A}$ as a linear transformation 
\[
\mathcal{A} \colon \mathbb{Q}^{n+1} \rightarrow  \mathbb{Q}^{2}, 
\]
let $\mathrm{Ker}(\mathcal{A})$ denote its kernel and let $\mathscr{L}$ denote 
\[
\mathscr{L} = \mathrm{Ker}(\mathcal{A}) \cap \mathbb{Z}^{n+1}.
\]
\begin{definition} \label{Sturmfels 3}
For an integer $v$ and rational number $u$, define
\[
\gamma(u,v) = 
\begin{cases}
&1  \text{ if } v =0\\
&(u)_{|v|} \text{ if } v <0\\ 
&0 \text{ if } u \in \mathbb{Z} \text{ and } 0>u\geq -v\\ 
 &\frac{1}{\prod_{i=1}^v (u+i)} \text{ otherwise}.
\end{cases}
\]
Note that if $u$ is not a negative integer, then
\[
\gamma(u,v) = \frac{u!}{(u+v)!}.
\]
For $u_i \in \mathbb{Q}, 0 \leq i \leq n$, define the series 
\[
[a_0^{u_0}a_1^{u_1}\ldots a_n^{u_n}] = \sum_{(v_0, \ldots, v_n) \in \mathscr{L}} \prod_{i=0}^n (\gamma(u_i,v_i) a_i^{u_i+v_i})
\]
which in this paper we call a ``bracket series".
Let $\xi$ denote a $d$-th root of $-1$.
Define 
\begin{equation}\label{X def}
X_{i_1,i_2, \xi} = \xi [a_{i_1}^{1/d} a_{i_2}^{-1/d}]+\frac{1}{d}\sum_{k=2}^d \xi^k [a_{i_1}^{(k-d)/d} a_{i_1+k-1} a_{i_2}^{-k/d} ]+ \frac{1}{d} [a_{i_1-1} a_{i_1}^{-1}]
\end{equation}
where $[a_{i_1-1} a_{i_1}^{-1}]$ denotes 0 if $i_1=0$.
\end{definition}

We also use the notation 
\[
\Sigma^* = \sum_{i=0, \neq i_1,i_2}^n \, \text{ and } \Pi^* = \prod_{i=0, \neq i_1,i_2}^n
\]
\begin{lemma} \label{l vi12}
 Suppose 
\[
 (v_0, \ldots, v_n) \in \mathrm{Ker}(\mathcal{A}).
\]
Then 
\begin{align}
v_{i_2} &= -\frac{1}{d}\Sigma^*  (i-i_1)v_i \label{vi12 2}\\
v_{i_1} &= -\frac{1}{d}\Sigma^* (i_2-i)v_i \label{vi12 1}\\ 
&= \frac{1}{d}(\Sigma^* (i-i_1)v_i)-\Sigma^* v_i  \label{vi12 1 rearranged}
\end{align}
\end{lemma}
\begin{proof}
By definition 
\begin{align*}
i_1v_{i_1}+ i_2 v_{i_2} &= -\Sigma^* i v_i\\ 
v_{i_1}+ v_{i_2} &= -\Sigma^*  v_i,
\end{align*}
Solving this system for $v_{i_1}$ and $v_{i_2}$ yields equations \eqref{vi12 2} and \eqref{vi12 1}. Adding and subtracting $i_1$ to $i_2$ in the sum of \eqref{vi12 1} and simplifying yields equation \eqref{vi12 1 rearranged}. This completes the proof. 
\end{proof}


 \begin{lemma} \label{l X power series}
1. With notation as above, the series $X_{i_1,i_2,\xi}$ is a power series in the variables $a_i$ for $0\leq i \leq n, i\neq i_1,i_2$ (that is, the exponents of these variables are non-negative integers) whose coefficients depend on $a_{i_1}$ and $a_{i_2}$. 

2. The term  
\begin{equation}\label{X power series term}
\Pi^* a_i^{n_i}. 
 \end{equation}
 can appear in at most one or two bracket series in definition \eqref{X def}, depending on whether 
 \[
 \Sigma^* (i-i_1)n_i
\]
is not or is equal to $-1\mod d$, respectively.  
 \end{lemma}
\begin{proof}
Note that each $a_i, 0\leq i \leq n, i\neq i_1,i_2$ appears in with corresponding $u_i$ equal to 0 or 1. Negative exponents can arise in two cases 
\begin{align*} 
 &a) \,\,u_i=0 \text{ and } v_i<0 \\ 
 &b) \,\,u_i=1 \text{ and } v_i<-1. 
 \end{align*}
In either case $\gamma(u_i, v_i)=0$. This proves statement 1.

In the definition \eqref{X def}, there are $d+1$ bracket series. If $d =1$, then there are exactly two bracket series in \eqref{X def}, and the theorem is true. Therefore assume $d \geq 2$. Consider the possible $u_i$ for $i \neq i_1,i_2$. We say that the $k$-th bracket series has $u_{i_1+k} =1$ for $k=0, 2\leq k\leq d$, except in the first bracket series in which $u_i=0$ for all $i \neq i_1,i_2$. 
Let $C$ denote the number 
\[
C = \Sigma^* (i-i_1)n_i.
\]
Suppose $C \equiv k-1 \mod d$ for some $k$. Now in each bracket series
\[
u_i+v_i = n_i
\]
for $i \neq i_1,i_2$, and $\vect{v} \in \mathscr{L}$. For $0 \leq j \leq d$, we thus define the $(n+1)$-tuple $\vect{v_j}$
\[
\vect{v_j} = (v_0, v_1, \ldots, v_n)
\]
where for $i \neq i_1, i_2$
\[
v_i =n_i - 1(i =i_1+j-1 \text{ and } j \neq 1)
\]
and $v_{i_1}$ and $v_{i_2}$ are determined by Lemma \ref{l vi12}:
\begin{align*}
v_{i_2} &= -\frac{1}{d}(C-j+1)\\ 
v_{i_1}&= \frac{1}{d}(C-j+1)+1(j\neq 1)-\Sigma^* n_i, 
\end{align*}
If the term \eqref{X power series term} appears in the $j$-th bracket series, then $\vect{v}_j \in \mathscr{L}$. The numbers $v_{i_2}$ and $v_{i_1}$ are integers exactly when $j\equiv k \mod d$. Thus if $k \not\equiv 0 \mod d$ the term \eqref{X power series term} may appear only in the $k$-th bracket series (assuming $1
\leq k \leq d-1$), and if $k\equiv 0\mod d$, it may appear in either the $0$-th or $d$-th bracket series. This completes the proof.

\end{proof}

\begin{theorem} \label{t recovery}
For integers $n_i \geq 0,  i \neq i_1,i_2$ not all 0, the coefficient of the term 
 \begin{equation}\label{recovery term}
\Pi^* a_i^{n_i}
 \end{equation}
in the series $X_{i_1,i_2,\xi}$ is 
\begin{equation}\label{recovery coeff} 
\frac{\xi^k}{d} (-1)^M (\frac{a_{i_1}}{a_{i_2}})^{(C+1)/d} \frac{(\frac{C+1}{d}-1)_{\sum^* n_i-1}}{a_{i_1}^{\sum^* n_i}\prod^* n_i!}
\end{equation}
where 
\begin{align*}
C &= \Sigma^* (i-i_1)n_i\\ 
&=k-1+M d
\end{align*}
for some integers $M$ and $0\leq k \leq d-1$.
\end{theorem}
\begin{proof}
We first assume $d \geq 2$. Consider the term \eqref{recovery term}. 

We first compute 
\[
\Pi^* \gamma(u_i,v_i).
\]
Since $u_i+v_i = n_i$, we check in either case  $u_i=0$ or $u_i=1$ that 
\[
 \gamma(u_i,v_i) = \frac{1}{n_i!}.
\]
Next let $C$ denote the number
 \[
 C=\Sigma^* (i-i_1)n_i.
 \]
We compute $ \gamma(u_{i_1},v_{i_1}) \gamma(u_{i_2},v_{i_2})$ depending on the equivalence class of $C \mod d$.

Suppose $C \equiv k-1 \mod d$ for some $k, 2 \leq k \leq d-1$ and write 
\[
C = k-1+M d
\]
for some integer $M$. Let $v_{i_1+k-1} =n_{i_1+k-1} - 1$ and let $v_i=n_i$ for all other $i \neq i_1,i_2$. These $v_i$ determine numbers $v_{i_1}$ and $v_{i_2}$ by Lemma \ref{l vi12}: 
\begin{align*}
v_{i_2} &= -\frac{1}{d}(C-k+1)\\ 
&=-M\\
v_{i_1}&= \frac{1}{d}(C-k+1)+1-\Sigma^* n_i \\ 
&=M+1-\Sigma^* n_i.
\end{align*}
Thus $v_{i_1}$ and $v_{i_2}$ are integers and $(v_0, \ldots v_n) \in \mathscr{L}$, and so the term \eqref{recovery term} appears in the bracket series of $k$-th term in the sum from definition of $X_{i_1,i_2,\xi}$. In that series we have $\displaystyle u_{i_1}=\frac{k}{d}-1$ and $\displaystyle u_{i_2}=-\frac{k}{d}$. 
Thus 
\begin{align}
u_{i_2}+v_{i_2} &= -\frac{1}{d}(C+1) \nonumber \\ 
u_{i_1}+v_{i_1}&= \frac{1}{d}(C+1)-\Sigma^* n_i \label{u+v}
\end{align}
and 
\begin{align*}
\gamma(u_{i_2}, v_{i_2}) \gamma(u_{i_1}, v_{i_1}) &= \left( \frac{(-k/d)!}{(-k/d-M)!}\right)  \left( \frac{(k/d-1)!}{(k/d+M-\Sigma^* n_i)!}\right) \\ 
&=   \frac{(-k/d)!}{(-k/d-M)!}\frac{(k/d)!}{(k/d+M)!} \frac{d}{k}(k/d+M)_{\Sigma^* n_i} \\ 
&= \frac{\sin(\pi (k/d+M))}{\sin(\pi k/d)}\frac{(k/d+M)_{\Sigma^* n_i}}{k/d+M}\\ 
&= (-1)^M (k/d+M-1)_{\Sigma^* n_i-1}
\end{align*}
where have used the identities for $x \in \mathbb{C}, m \in \mathbb{Z}$
\begin{align*}
x!(-x)! &= \frac{\pi x}{\sin(\pi x)}\\ 
\sin(\pi(x+m)) &=\sin(\pi x) (-1)^m
\end{align*}
and the fact that $\Sigma^*n_i$ is a positive integer. 
Putting this together, we obtain that the coefficient of term \eqref{recovery term} is equal to expression \eqref{recovery coeff}. 
This proves the theorem for this case.

Now suppose $C =M d$ for some integer $M$. Let $v_i=n_i$ for all $i \neq i_1,i_2$. These $v_i$ determine numbers $v_{i_1}$ and $v_{i_2}$ by Lemma \ref{l vi12}: 
\begin{align*}
v_{i_2} &= -\frac{C}{d}\\ 
&=-M\\
v_{i_1}&= \frac{C}{d}-\Sigma^* n_i\\ 
&=M-\Sigma^* n_i.
\end{align*}
Thus $v_{i_1}$ and $v_{i_2}$ are integers and $(v_0, \ldots v_n) \in \mathscr{L}$, and so the term \eqref{recovery term} appears in the bracket series $\displaystyle [a_{i_1}^{1/d} a_{i_2}^{-1/d} ]$ in the definition of $X_{i_1,i_2,\xi}$. In that series we have $\displaystyle u_{i_1}=\frac{k}{d}-1$ and $\displaystyle u_{i_2}=-\frac{k}{d}$ thus equations \eqref{u+v} still hold. Using similar reasoning above we obtain 
\[
\gamma(u_{i_2}, v_{i_2}) \gamma(u_{i_1}, v_{i_1})= \frac{(-1)^M}{d} (1/d+M-1)_{\sum_{i}^* n_i-1}.
\]
This proves the theorem in this case.

Now suppose $C =-1+M d$ for some integer $M$. 
First let $v_{i_1+d-1} =n_{i_1+d-1}-1$ and $v_i = n_i$ fir all other $i\neq i_1,i_2$. These $v_i$ determine numbers $v_{i_1}$ and $v_{i_2}$ by Lemma \ref{l vi12}: 
\begin{align*}
v_{i_2} &= -\frac{1}{d}(C-d+1)\\ 
&=-M+1\\
v_{i_1}&= \frac{1}{d}(C-d+1)+1-\Sigma^* n_i \\ 
&=M-\Sigma^* n_i.
\end{align*}
Thus $v_{i_1}$ and $v_{i_2}$ are integers and $(v_0, \ldots v_n) \in \mathscr{L}$, and so the term \eqref{recovery term} appears in the bracket series $\displaystyle [a_{i_1+d-1}, a_{i_2}^{-1}]$ in the definition of $X_{i_1,i_2,\xi}$. In this series, $u_{i_1}=0$ and $u_{i_2}=-1$, so equations \eqref{u+v} still hold. We check that if $M\geq 1$ then

\begin{align*}
\gamma(u_{i_2}, v_{i_2}) \gamma(u_{i_1}, v_{i_1})&= \left((-1)^{M-1}(M-1)!\right) \left( \frac{1}{(M-\Sigma^*n_i)!}\right) \\ 
&=(-1)^{M-1}(M-1)_{\Sigma^* n_i-1},  
\end{align*}
and if $M\leq 0$ then $\gamma(u_{i_2}, v_{i_2}) \gamma(u_{i_1}, v_{i_1}) = 0$. Note that if $i_1=0$, then here we must have $M \geq 1$. 

The next case occurs only if $i_1\geq 1$, so assume that. Let $v_{i_1-1} =n_{i_1-1}-1$ and $v_i = n_i$ fir all other $i\neq i_1,i_2$. These $v_i$ determine numbers $v_{i_1}$ and $v_{i_2}$ by Lemma \ref{l vi12}: 
\begin{align*}
v_{i_2} &= -\frac{1}{d}(C+1)\\ 
&=-M\\
v_{i_1}&= \frac{1}{d}(C+1)+1-\Sigma^* n_i \\ 
&=M+1-\Sigma^* n_i.
\end{align*}
Thus $v_{i_1}$ and $v_{i_2}$ are integers and $(v_0, \ldots v_n) \in \mathscr{L}$, and so the term \eqref{recovery term} appears in the bracket series $\displaystyle [a_{i_1-1}, a_{i_1}^{-1}]$ in the definition of $X_{i_1,i_2,\xi}$. In this series, $u_{i_2}=0$ and $u_{i_1}=-1$, so equations \eqref{u+v} still hold. We check that if $M\leq 0$ then

\begin{align*}
\gamma(u_{i_2}, v_{i_2}) \gamma(u_{i_1}, v_{i_1})&=\left((-1)^{|M|+\Sigma^*n_i-1}(|M|+\Sigma^*n_i-1)!\right) \left( \frac{1}{|M|!}\right)\\ 
&=(-1)^{M}(M-1)_{\Sigma^* n_i-1},
\end{align*}
and if $M\geq 1$ then $\gamma(u_{i_2}, v_{i_2}) \gamma(u_{i_1}, v_{i_1})=0$.
This completes the proof for $d \geq 2$.

For $d=1$, the definition \eqref{X def} has exactly two bracket series (and only if $i_1=0$), and the proof proceeds similarly to the above case when $C =-1+M d$.
This completes the proof.
\end{proof}
\begin{corollary} \label{c recovery}
Suppose $c_{i_1}$ and $c_{i_2}$ are non-zero complex numbers and parametrize $\mathbb{C}^{n-1}$ with the coordinate variables $c_i, i \neq i_1,i_2$. 
Let 
\begin{align*}
g(z) &= 1+\frac{c_{i_2}}{c_{i_1}}z^d\\
f(z) &= g(z) + \Sigma^* c_i z^{i-i_1}
\end{align*}
and $\phi(\vect{c}) \colon U \rightarrow L$ a smooth function where $U \in \mathbb{C}^{n-1}$ is a neighborhood of $\vect{0}$ such that 
\begin{align*}
f(\phi(\vect{c})) &= 0\\
\phi(\vect{0}) &= \alpha
\end{align*}
where $g(\alpha)=0$. 
Denote 
 \[
\frac{\vect{c}}{c_{i_1}} = (\frac{c_0}{c_{i_1}}, \frac{c_1}{c_{i_1}},\ldots, \frac{c_{i_1-1}}{c_{i_1} }, \frac{c_{i_1+1}}{c_{i_1} }, \ldots, \frac{c_{i_2-1}}{c_{i_1} }, \frac{c_{i_2+1}}{c_{i_1} }, \ldots, \frac{c_n}{c_{i_1} }) \in \mathbb{C}^{n-1}.
\]

Then as formal series, the Taylor series of $\phi(\frac{\vect{c}}{c_{i_1}} )$ about $\vect{0}$ with respect to the variables $c_i, i \neq i_1,i_2$ is equal to $X_{i_1,i_2, \xi}$, where we identify $c_i=a_i$
and
\[
\alpha =\xi \frac{a_{i_1}^{1/d}}{a_{i_2}^{1/d}}.
\]
\end{corollary}
\begin{proof} 
This follows from comparing the coefficient of Theorem \ref{t recovery} with the formula in Theorem \ref{t main multiset} with $\beta=d$ and $\vect{\gamma}$ the $(n-1)$-tuple 
\[
\vect{\gamma}= (-i_1, 1-i_1,\ldots, -1, 1, \ldots, d-1, d+1, \ldots, n-i_1).
\]
\end{proof}

\section{Alternative proof for $\beta=1$}\label{s alt}

Suppose in Theorem \ref{t main multiset} that $\beta=1$. From the proof of that theorem, it is sufficient to prove Theorem \ref{t nu} in the case $\nu=1$. We therefore give an alternative proof of that theorem assuming $\nu=1$. A special case of Theorem \ref{t s tau} implies this result. First we prove the following two lemmas.

\begin{lemma} \label{fall identity}
For indeterminates $a$ and $b$ and an integer $n \geq 0$, 
\[
(a+b)_n = \sum_{i=0}^n {n \choose i} (a)_i(b)_{n-i}
\]
and equivalently 
\[
{a+b \choose n} = \sum_{i=0}^n {a \choose i} {b \choose n-i}.
\]
\end{lemma}
\begin{proof}
Apply the product rule $n$ times in 
\[
(\frac{d}{dt})^{n} (t^a) (t^b)
\]
and evaluate at $t=1$
to give the right side. Taking the $n$-th derivative of 
\[
t^{a+b}
\] 
without using the product rule gives the left side. This completes the proof. 
\end{proof}

\begin{lemma} \label{l Newton coeff}
For integers $i,n \geq 0$ and an indeterminate $b$,
\begin{equation} \label{Newton coeff}
\sum_{r=0}^i (-1)^{i-r} {i \choose r}{b + r\choose n} = {b\choose n-i}
\end{equation}
\end{lemma}
\begin{proof}
For any function $f(a)$, the coefficient of $\displaystyle {a \choose i}$ in the Newton series of $f(a)$ is 
\[
\sum_{r=0}^i (-1)^{i-r} {i \choose r}f(r).
\] 
Consider the expression 
\begin{equation} \label{a function}
{a+b \choose n}
\end{equation}
as a function of $a$. Therefore the left side of equation \eqref{Newton coeff} is the coefficient of $\displaystyle {a \choose i}$ in the Newton series of the function \eqref{a function}. From Lemma \ref{fall identity}, this coefficient is also equal to 
\[
{b \choose n-i}.
\]
This completes the proof.
\end{proof}

Now take the left side of Theorem \ref{t nu} when $\nu=1$ and apply Lemma \ref{l Newton coeff} with $b=\sum_{i=1}^{N} x_i$, $n=N-1$, and $i=k-1$ to yield
\[
{N-1 \choose k-1} (\sum_{i=1}^{N} x_i)_{N-1}.
\]
Rename $k$ by $N-j$. We thus seek to prove  
\begin{theorem} \label{t nu=1}For integers $0\leq j\leq N-1$ and indeterminates $x_i, 1 \leq i \leq N$, 
\begin{equation}\label{nu=1}
{N-1 \choose j} (\sum_{i=1}^{N} x_i)_{N-1}= \sum_{s \in S(N,N-j)} \prod_{i=1}^{N-j} (\sum_{m\in s_i} x_m)_{|s_i|-1}.
\end{equation}
\end{theorem}

We now present the notation and definitions to state Theorem \ref{t s tau} which will imply Theorem \ref{t nu=1}. Let $N$ and $j$ be integers $0\leq j < N$. For an integer $k$, $0\leq k \leq N-j$, let $\vect{u}$ denote a $k$-tuple of non-negative integers 
\[
\vect{u} = (u_1, \ldots, u_k). 
\]
Let $R$ be a commutative ring and $\delta: R \rightarrow R$ a derivation. Let $\vect{f}$ denote an $N$-tuple 
\[
\vect{f} = (f_1, \ldots, f_N)
\]
where $f_h \in R$. We use the same notation $f_i^{(n)}$ as in the proof of Theorem \ref{t deriv set}. Let $\tau$ be a set of $k$ integers 
\[
\tau = (\tau(1), \ldots, \tau(k))
\]
where $1 \leq \tau(i) < \tau(i+1) \leq N-j$, and denote the set of such $\tau$ by $T(N-j,k)$. We also view a $\tau$ as a strictly increasing function 
\[
\tau \colon [1,k]  \rightarrow [1,N-j].
\]
We require that $T(N-j,0)$ consists of one element, the empty set. 
Define $F(j,\vect{f}, \vect{u})\in R$ by
\begin{align}
F(j,\vect{f}, \vect{u})=&\sum_{\tau \in T(N-j,k)}\sum_{s \in S(N,N-j)}(\prod_{i=1}^k {|s_{\tau(i)}|-1 \choose u_i} \delta^{|s_{\tau(i)}|-1-u_i} \prod_{h\in s_{\tau(i)}} f_h) \nonumber \label{F def}\\
&\times (\prod_{i\notin \tau}  \delta^{|s_i|-1} \prod_{h\in s_i} f_h) \nonumber \\
\end{align}
where a term is 0 if $|s_{\tau(i)}|-1-u_i<0$, as the binomial coefficient is equal to 0.
Define the number $C(j,N,\vect{u})$ by 
\[
C(j,N,\vect{u})=\frac{(N-1)!(N-j+\sum_{i=1}^k u_i)}{(N-j-k)!(j - \sum_{i=1}^k u_i)! \prod_{i=1}^k (u_i)! \prod_{i=1}^k(k-i+1+ \sum_{g=i}^k u_g)}. 
\]

\begin{theorem} \label{t s tau}
With the above definitions
\begin{equation} \label{s u}
F(j,\vect{f}, \vect{u}) = C(j,N,\vect{u})\delta^{j - \sum_{i=1}^k u_i}\prod_{i=1}^N f_i.  
\end{equation}
\end{theorem} 
\begin{proof}
We use induction on $N$. It is straightforward to prove the theorem in the case $N=1$ and $k=0$ or $k=1$. For an $N\geq2$, assume the theorem is true for all values less than $N$. We will prove the induction step by comparing coefficients of both sides of equation \eqref{s u} of a term of the form 
\begin{equation} \label{s tau term}
\prod_{i=1}^N f_i^{(n_i)}.
\end{equation}
For this term to appear on either side of equation \eqref{s u} we must have 
\begin{equation} \label{n sum j k}
\sum_{i=1}^N n_i = j - \sum_{i=1}^k u_i,
\end{equation}
so at least one $n_i$ must be 0. By the symmetry of $F(j,\vect{f}, \vect{u})$ in $f_i$ (Lemma \ref{l f sym}), we may assume that the term \eqref{s tau term} is of the form 
\begin{equation} \label{s tau term l}
\prod_{i=1}^l f_i^{(n_i)}.
\end{equation}
Let $M(j,N,\vect{u})$ denote the set of $N$-tuples $\vect{n}$ of non-negative integers
\[
\vect{n} = (n_1, \ldots, n_N)
\]
satisfying \eqref{n sum j k}. For $X \in [1,N]$ and $\vect{n} \in M(j,N,\vect{u})$, define $\mathrm{degsum}(X, \vect{n})$ to be 
\[
\mathrm{degsum}(X,\vect{n}) = \sum_{i \in X } n_i.
\]
For an $r$-tuple of integers $\vect{m}=(m_i)_{i=1}^r$ denote the multinomial coefficient $\mathrm{mult}(\vect{m})$ 
\[
\mathrm{mult}(\vect{m}) = \frac{(\sum_{i=1}^r m_i)!}{\prod_{i=1}^r m_i!}.
\] 
Note that $\mathrm{mult}(\vect{m})$ is independent of the ordering of $\vect{m}$. 

Expand the left side of equation \eqref{s u} to obtain
\begin{equation}  \label{expansion}
\sum_{\vect{n}\in M(j,N,\vect{u})}  \sum_{\tau \in T(N-j,k)} \sum_{s \in S(N,N-j)} \mathbf{1}(s,\tau, \vect{n}, \vect{u}) \left(\prod_{i=1}^k {|s_{\tau(i)}|-1 \choose u_i} \right)\left(\prod_{i=1}^{N-j} \mathrm{mult}((n_x)_{x\in s_i}) \right)\prod_{i=1}^N f_i^{(n_i)}
\end{equation}
where $\mathbf{1}(s,\tau, \vect{n}, \vect{u}) = 1$ if the following conditions are satisfied 
\begin{align}
|s_{\tau(i)}|-1-u_i &= \mathrm{degsum}(s_{\tau(i)}, \vect{n}) \text{ for } 1 \leq i \leq k \nonumber \\ 
|s_i|-1 &= \mathrm{degsum}(s_i, \vect{n}) \text{ for } i \notin \tau \label{condition}
\end{align}
and 
$\mathbf{1}(s,\tau, \vect{n}, \vect{u}) = 0$ otherwise. 

For an integer $l,1\leq l \leq N-1$ and an integer $v$ with $0 \leq v \leq l$, let $\sigma \in S(l,v)$ 
\[
\sigma = (\sigma_1, \ldots, \sigma_v).
\] 
where we recall the subsets $\sigma_i$ are ordered such that 
\[
\min(\sigma_i) < \min(\sigma_{i+1}).
\]
For a $\tau \in T(N-j,k)$, suppose 
\begin{align*}
\tau(h) &\leq v\\
\tau(h+1) &> v.
\end{align*}
Then we write $\tau$ as 
\[
(\tau_1, \tau_2)
\]
where $\tau_1 \in T(v,h)$ and $\tau_2 \subset [v+1,N]$ and $|\tau_2|=k-h$.  
For a $\sigma \in S(l,v)$, let $S(N,N-j; \sigma) \subset S(N,N-j)$ be the subset consisting of those $s$ such that 
\[
v \leq \mathrm{length}(s)
\]
and 
\[
\sigma_i \subset s_i \quad \quad \text{ for } 1 \leq i \leq v.
\]

Now if $s \in S(N,N-j;\sigma), \mathbf{1}(s,\tau, \vect{n}, \vect{u}) \neq 0$, and $\tau=(\tau_1,\tau_2)$ as above, then the quantity 
\[
\prod_{i=1}^k {|s_{\tau(i)}|-1 \choose u_i}  
\]
is equal to 
\[
\prod_{i=1}^h {u_i + \mathrm{degsum}(\sigma_{\tau(i)},\vect{n}) \choose u_i}
\]
and depends only on $\sigma,\tau_1$ and $\vect{n}$.

Define $M(j,N,\vect{u};l) \subset M(j,N,\vect{u})$ to be the subset consisting of those $\vect{n}$ such that $n_i = 0$ for $i >l$.
In expression \eqref{expansion} consider the sub-sum 
\begin{align}
&\sum_{\vect{n}\in M(j,N,\vect{u};l)} \left(\prod_{i=1}^N f_i^{(n_i)}(t) \right )\sum_{v=1}^l \sum_{\sigma \in S(l,v)} \left ( \prod_{i=1}^v \mathrm{mult}(\{n_x\}_{x\in \sigma_i})  \right ) \sum_{h=1}^v \sum_{\tau_1 \in T(v,h)} \nonumber \\ 
&\left( \prod_{i=1}^h {u_i + \mathrm{degsum}(\sigma_{\tau_1(i)},\vect{n}) \choose u_i}  \right) \sum_{\tau_2 \subset [h+1,V], |\tau_2|= k-h}\sum_{s \in S(N,N-j;\sigma)} \mathbf{1}(s,(\tau_1,\tau_2), \vect{n}, \vect{u}). \nonumber\\ \label{subsum}
\end{align}
We claim that the innermost sum evaluates to
\begin{align*}
&\sum_{\tau_2 \subset [h+1,V], |\tau_2|= k-h} \sum_{s \in S(N,N-j; \sigma)} \mathbf{1}(s,(\tau_1,\tau_2), \vect{n}, \vect{u}) \\ 
&={N-l \choose j+v+k-h-l} {j+v+k-h-l \choose j+v-l-\sum_{i=h+1}^k u_i}  \mathrm{mult}(\epsilon)\prod_{i=h+1}^k {k-i+\sum_{g=i}^k u_g \choose u_i}
\end{align*} 
where $\epsilon$ is the $v$-tuple
\[
\epsilon = ((\mathrm{degsum}(\sigma_{\tau_1(i)}) +u_i+1-|\sigma_{\tau_1(i)}| )_{i=1}^h, ( \mathrm{degsum}(\sigma_i) +1-|\sigma_i| )_{i \notin \tau_1, i \leq v}).
\]
We prove this claim now by constructing all pairs $(s, \tau_2)$ that have a non-zero contribution to the sum. No matter what the choices of $s$ and $\tau_2$ are, the conditions \eqref{condition} imply 
\begin{align} \label{union}
|\left(\bigcup_{i=1}^v s_i \right) \cup \left(\bigcup_{i=1}^{k-h} s_{\tau_2(i)} \right)|&= \sum_{i=1}^l n_i + \sum_{i=1}^ku_i+v+k-h\\ 
&= j+v+k-h.\nonumber
\end{align}
By construction the elements $1, \ldots, l$ are in 
\[
\bigcup_{i=1}^v s_i,
\]
 so we have to choose $j+v+k-h-l$ elements from a set of order $N-l$ to be the remaining elements in the union \eqref{union}. Let $Z$ denote the set of these $j+v+k-h-l$ elements. Next, from $Z$ we must choose 
\[
j+v-l-\sum_{i=1}^h u_i
\]
elements to fill the $s_i, 1 \leq i \leq v$, and there are $\mathrm{mult}(\epsilon)$ ways to do this. Now from the set $Z_1$ consisting of
\[
k-h+\sum_{g=1}^{k-h}u_{g+h}
\]
elements not chosen yet, we construct $k-h$ sets $w_i, 1 \leq i\leq k-h$. We place the element $\min(Z_1)$ in $w_1$ and choose $u_{h+1}$ other elements to be in $w_1$. Proceeding in this manner, from a set $Z_i$ of order 
\[
k-h-i+1+\sum_{g=i}^{k-h}u_{g+h},
\] 
we construct $w_i$ to consist of $\min(Z_i)$ and a choice of $u_{i+h}$ other elements from $Z_i$. Now let $s$ be the partition consisting of the sets 
\begin{align*}
 &s_i, \text{ for } 1 \leq i \leq v \\ 
 &w_i, \text{ for } 1 \leq i \leq k-h
\end{align*}
and singleton sets, and let $\tau_2$ be determined by 
\[
\tau_2(i) = i'
\]
where $s_{i'}=w_i$. This proves the claim.

 Putting this claim into the expression \eqref{subsum} and simplifying, we get that the coefficient of $\displaystyle \prod_{i=1}^N f_i^{(n_i)}(t) $ for $\vect{n} \in M(j,N,\vect{u};l)$ is 
\begin{align}
&\frac{1}{\prod_{i=1}^k u_i! \prod_{i=1}^l n_i!} \sum_{v=1}^l \sum_{\sigma \in S(l,v)} \sum_{h=1}^v \sum_{\tau_1 \in T(v,h)} {N-l \choose j+v+k-h-l} \nonumber\\ 
&\times \frac{(j+v+k-h-l)!}{\prod_{i=1}^{k-h} (k-h-i+1+\sum_{g=h+i}^{k}u_g)} Y(\sigma, \tau_1,\vect{n},\vect{u}) \nonumber \\ \label{subsum simplified}
\end{align}
where 
\[
Y(\sigma, \tau_1,\vect{n},\vect{u}) = \prod_{i=1}^{|\tau_1|} (u_i + \mathrm{degsum}(\sigma_{\tau_1(i)}))_{|\sigma_{\tau_1(i)}|-1}  \prod_{i\notin \tau_1,1\leq  i \leq \mathrm{length}(\sigma)} ( \mathrm{degsum}(\sigma_i))_{|\sigma_i|-1}.
\]

The corresponding coefficient on the right side of equation \eqref{s u} is 
\[
C(j,N,\vect{u}) \mathrm{mult}(\vect{n})
\]
which simplifies to 
\begin{equation} \label{coeff right}
 \frac{1}{\prod_{i=1}^k u_i! \prod_{i=1}^l n_i!} {N-1 \choose j+k-1}\frac{(N-j+\sum_{i=1}^k u_i)(j+k-1)!}{\prod_{i=1}^k (k-i+1+\sum_{g=i}^k u_g)}.
\end{equation} 
Equating expressions \eqref{subsum simplified} and \eqref{coeff right} yields the equation 
\begin{align}
& \sum_{v=1}^l \sum_{\sigma \in S(l,v)} \sum_{h=1}^v \sum_{\tau_1 \in T(v,h)} {N-l \choose j+v+k-h-l} (j+v+k-h-l)! \prod_{i=1}^h (k-i+1+\sum_{g=i}^k u_g) \nonumber \\ 
&\times Y(\sigma, \tau_1,\vect{n},\vect{u})\nonumber \\ 
&= {N-1 \choose j+k-1} (N-j+\sum_{i=1}^k u_i)(j+k-1)!.\nonumber \\ \label{equate} 
\end{align}
We prove equation \eqref{equate} by comparing coefficients in the basis of the binomial functions in $N$.
To express the right side of equation \eqref{equate} in that basis, we apply the following identities for any non-negative integers $n$ and $d$
\[
{x \choose n}(x-n) = { x \choose n+1}(n+1)
\]
and 
\[
{ x \choose n} = \sum_{i=0}^d {d \choose i }{x-d \choose n-i}.
\]
Using $x=N-1, n=j+k-1,$ and $d=l-1$, we obtain that the coefficient of 
\begin{equation} \label{N bin}
{N-l \choose j+k-a}
\end{equation}
where $0\leq a \leq l$ is 
\[
(j+k-1)!\left ( (j+k){l-1 \choose a}+ (k+\sum_{i=1}^k u_k){l-1 \choose a-1 }\right).
\]
On the left side of equation \eqref{equate}, we have a contribution to the coefficient of the binomial \eqref{N bin} when $v-h=l-a$. The total coefficient is thus
\begin{align}
&  (j+k-a)!\sum_{v=l-a}^l  \left(\prod_{i=1}^{a+v-l} (k-i+1+\sum_{g=i}^k u_g)  \right) \sum_{\tau_1 \in T(v,a+v-l)} \sum_{\sigma \in S(l,v)} Y(\sigma, \tau_1,\vect{n},\vect{u}).\nonumber \\ 
\end{align}
We equate these two coefficients, divide by $(j+k-a)!$ and define $X$ by
\[
X = k+ \sum_{i=1}^{k}u_i
\]
and $U(n,\vect{u})$ by
\[
U(n,\vect{u}) = \prod_{i=1}^n (X-i+1-\sum_{g=1}^{i-1}u_g).
\]
We now view $X, u_i$ and $n_i$ as indeterminates  and see that the induction step is implied by the following polynomial identity, for all pairs of integers $0\leq a \leq l$: 
\begin{align}
&\sum_{v=l-a}^l  U(v-l+a, \vect{u}) \sum_{\sigma \in S(l,v)} \sum_{\tau_1 \in T(v,v-l+a)} Y(\sigma,\tau_1,\vect{n},\vect{u}) \nonumber \\ 
& = {l-1 \choose a}(X + \sum_{i=1}^l n_i)_a+ X{l-1 \choose a-1}(X -1+ \sum_{i=1}^l n_i)_{a-1}. \nonumber \\ \label{X U}
\end{align}


By the induction hypothesis and Lemma \ref{u independent}, we have that the left side is equal to 
\[
\sum_{v=l-a}^l \frac{(l-1)!v}{(l-a)!(l-v)!(v-l+a)!} (X)_{v-l+a}(\sum_{i=1}^l n_i)_{l-v}.
\]
Re-index this sum by $v=l-a+r$ to obtain 
\[
\sum_{r=0}^a \frac{(l-1)!(l-a+r)}{(l-a)!(a-r)!r!} (X)_{r}(\sum_{i=1}^l n_i)_{a-r},
\]
which is equal to 
\[
{l-1 \choose a} \sum_{r=0}^a {a \choose r}(X)_r (\sum_{i=1}^l n_i)_{a-r}+ {l-1 \choose a-1}  \sum_{r=0}^a {a-1 \choose r-1}  (X)_r (\sum_{i=1}^l n_i)_{a-r}.
\]
The second sum is equal to 
\[
X \sum_{r'=0}^{a-1} (X-1)_{r'}  (\sum_{i=1}^l n_i)_{a-1-r'}
\]
where have re-indexed $r=r'+1$. Now apply identity \eqref{fall identity} to obtain the right side of \eqref{X U}. This completes the induction step and the proof. 
\end{proof}

Now we can prove Theorem \ref{t nu=1}. 
\begin{proof}
Apply Theorem \ref{t s tau} with $k=0$, $R$ the ring of polynomials in $t$, $\delta$ differentiation with respect to $t$, and $f_i = t^{x_i}$ where $x_i$ is any non-negative integer, and then set $t=1$. This shows that both sides of equation \eqref{nu=1} are equal when $x_i$ are non-negative integers, and since both sides are polynomials in $x_i$, they must be equal as polynomials. This completes the proof.
\end{proof}

\begin{lemma} \label{l f sym}
Assume Theorem \ref{t s tau} is true for all values of $N$ less than some $N_0$. For any $N_0$-tuple $\vect{f}$ of elements in $R$, the expression $F(j, \vect{f},\vect{u})$ is symmetric in the $f_i$. That is, if $\vect{g}$ is a reordering of the $N_0$-tuple $\vect{f}$, then 
\[
F(j,\vect{f},\vect{u}, ) = F(j,\vect{g},\vect{u}, ).
\]
\end{lemma}
\begin{proof}
It is sufficient to prove that $F(j, \vect{f},\vect{u},)$ is invariant under each  transposition $f_r(t) \leftrightarrow f_{r+1}(t)$. We fix an $r$ and consider a set partition $s$ in the sum \eqref{F def}. If $r$ and $r+1$ are in the same subset for this $s$, then the contribution is invariant under $r \leftrightarrow r+1$. Therefore suppose that $r \in s_a$ and $r+1 \in s_b$ for some $a\ \neq b$. If at least one the following  
\[
r \neq \min(s_a) \text{ or } r+1 \neq \min(s_b)
\]
holds, then let $s'$ denote the set partition obtained from $s$ by switching $r$ and $r+1$. Then 
\begin{align*}
s_a' = \{r+1\}\cup (s_a \setminus \{r\})\\ 
s_b' = \{r\}\cup (s_b \setminus \{r+1\}),
\end{align*}
so the expression 
\begin{equation} \label{D sym}
D(s,\tau, j, \vect{f}, \vect{u})+ D(s',\tau, j, \vect{f}, \vect{u})
\end{equation}
is invariant under $r \leftrightarrow r+1$.  

Now suppose that both 
\begin{equation} \label{both min}
r = \min(s_a) \text{ and } r+1 = \min(s_b)
\end{equation}
(so $b=a+1$) and that both $\{a, a+1\} \notin \tau$. Then the expression \eqref{D sym} is again invariant under $r \leftrightarrow r+1$. If $a \in \tau$ and $a+1 \notin \tau$, then let $\tau'$ be the function obtained from $\tau$ by making $\tau'(\tau^{-1}(a)) = a+1$. Then 
\begin{equation} \label{D sym 2}
D(s,\tau, j, \vect{f}, \vect{u})+ D(s',\tau', j, \vect{f}, \vect{u})
\end{equation}
 invariant under $r \leftrightarrow r+1$. The case $a +1\in \tau$ and $a \notin \tau$ is handled similarly.
 
Thus suppose equations \eqref{both min} hold and that both $a$ and $a+1 \in \tau$. Now if $u_{\tau^{-1}(a)} \neq u_{\tau^{-1}(a+1)}$, then the expression \eqref{D sym 2} is not invariant under $r \leftrightarrow r+1$, so we prove invariance another way: given integers $n_r$ and $n_{r+1}$, in the expansion of $F(j, \vect{f},\vect{u})$, we prove that the coefficient  of
\[
f_{r}^{(n_r)} f_{r+1}^{(n_{r+1})}
\]
is equal to the coefficient of 
\[
f_{r}^{(n_{r+1})}f_{r+1}^{(n_r)}.
\]
For an $s$ satisfying condition \eqref{both min}, let $E(s,a)$ denote the set of set partitions $s'$ obtained from $s$ such that 
\begin{align*}
s_a ' \cup s_{a+1}' &= s_a \cup s_{a+1}\\ 
s_i ' &= s_i \text{ for } i \neq a, a+1 \\ 
r &\in s_a' \text{ and } r+1 \in s_{a+1}'.
\end{align*}
We claim the expression 
\[
\sum_{s' \in E(s,a)}D(s',\tau, j, \vect{f}, \vect{u})
\]
is invariant under $f_r \leftrightarrow f_{r+1}$. In the expansion of this sum, consider a term of the form 
\[
m(\vect{f}, \vect{n},I)=\prod_{i \in I} f_i^{(n_i)} \prod_{i \notin I} f_i
\]
where 
\[
I \subset (s_a \cup s_{a+1}) \text{ such that } r,r+1 \in I,
\]
and 
\[
\vect{n} = (n_i)_{i \in I}
\]
is some $|I|$-tuple of integers where $n_i >0$ if $i \neq r,r+1$. Denote  $\tau^{-1}(a)$ by $y$. Now $D(s',\tau, j, \vect{f}, \vect{u})$ has a non-zero contribution to the coefficient of $m(\vect{f}, \vect{n},I)$ only if 
\begin{align}
|s_a'| - 1 - u_{y} &= \sum_{i \in I \cap s_a'} n_i  \nonumber \\ 
|s_{a+1}'| - 1 - u_{y+1} &= \sum_{i \in I \cap s_{a+1}'} n_i. \nonumber \\ \label{s' condition}. 
\end{align}
Adding the above two equations gives 
\[
|s_a' \cup s_{a+1}'|-2-u_y - u_{y+1} = \sum_{i \in I} n_i.
\]
Since at least one of $n_r$ and $n_{r+1}$ is greater than zero by assumption, and the remaining $n_i$ are positive by assumption, we therefore have the the coefficient of $m(\vect{f}, \vect{n},I)$ can be non-zero only if 
\[
|s_a' \cup s_{a+1}'|-2-u_y - u_{y+1} \geq |I|-1
\]
and thus 
\[
N\geq |s_a' \cup s_{a+1}'|>|I|.
\]  
Define
\[
\mathbf{1}(s',\tau, \vect{u}, \vect{n},I)=1
\]
if the conditions \eqref{s' condition} are satisfied and 0 otherwise. The coefficient of $m(\vect{f}, \vect{n},I)$ is thus 
\begin{equation} \label{m coeff}
B(s,\tau,\vect{u}) \sum_{s' \in E(s,a)}{|s_a'|-1 \choose u_y} {|s_{a+1}'|-1 \choose u_{y+1}} \mathrm{mult}((n_i)_{i \in I \cap s_a'})  \mathrm{mult}((n_i)_{i \in I \cap s_{a+1}'}) \mathbf{1}(s',\tau, \vect{u}, \vect{n},I).
\end{equation}
 where $B(s,\tau,\vect{u})$ is some number depending only on $s,\tau$ and $u$. 
 
Now let $S(I, 2;r,r+1)$ denote the set of set partitions of $I$ into two subsets such that $r$ and $r+1$ are not in the same set. It follows that the sum in expression \eqref{m coeff} is 
\begin{align*}
&\sum_{\sigma \in S(I,2;r,r+1)}{u_y+ \mathrm{degsum}(\sigma_1, \vect{n}) \choose u_y} {u_{y+1}+\mathrm{degsum}(\sigma_2, \vect{n}) \choose u_{y+1}}   \\ 
\times & \mathrm{mult}(\mathrm{degsum}(\sigma_1, \vect{n})+u_y+1-|\sigma_1|,\mathrm{degsum}(\sigma_2, \vect{n})+u_{y+1}+1-|\sigma_2|) \\ 
\times &\mathrm{mult}((n_i)_{i \in \sigma_1})  \mathrm{mult}((n_i)_{i \in \sigma_2})
\end{align*}
which simplifies to 
\[
\frac{(\sum_{i\in I}n_i+u_y+u_{y+1}+2-|I|)!}{u_y!u_{y+1}! \prod_{i\in I}n_i!} \sum_{\sigma \in S(I,2; r,r+1)} (u_y+ \mathrm{degsum}(\sigma_1, \vect{n}))_{|\sigma_1|-1}(u_{y+1}+ \mathrm{degsum}(\sigma_2, \vect{n}))_{|\sigma_2|-1}.
\]
The number multiplying this sum is clearly invariant under $n_r \leftrightarrow n_{r+1}$, so we must prove that the sum is also invariant. This sum is equal to 
\begin{align}
 &\sum_{\sigma \in S(I,2)} ( \mathrm{degsum}(\sigma_1, \vect{\mu}))_{|\sigma_1|-1}( \mathrm{degsum}(\sigma_2, \vect{\mu}))_{|\sigma_2|-1} \label{sum 1}\\ 
&-\sum_{\sigma \in S(l,2) \setminus S(I,2; r,r+1)} ( \mathrm{degsum}(\sigma_1, \vect{\mu}))_{|\sigma_1|-1}(u_{y+1}+ \mathrm{degsum}(\sigma_2, \vect{\mu}))_{|\sigma_2|-1}\label{sum 2}
\end{align}
where $\vect{\mu}$ is obtained from $\vect{n}$ by setting 
\begin{align*}
\mu_r &= u_y+ n_r \\ 
\mu_{r+1} &= u_{y+1}+ n_{r+1} \\
\mu_i &= n_i \text{ for all other } i \in I. 
\end{align*}
The sum at line \eqref{sum 2} is invariant under $n_r \leftrightarrow n_{r+1}$ because $r$ and $r+1$ are always in the same subset. The sum at line \eqref{sum 1} is equal to 
\begin{equation} \label{F g}
F(|I|-2,\vect{g},\emptyset)|_{t=1}
\end{equation}
 where $R$ is the ring of polynomials in $t$, $\delta$ is differentiation with respect to $t$, $g_1 = t^{n_r+u_y}, g_2= t^{n_{r+1}+u_{y+1}}$, and $g_i = t^{n_{i'}}$ where $i'$ is the $i$-th smallest element of $I$, for $i> 2$. Since $|I|<N_0$, we may apply the assumption on Theorem \ref{t s tau} to see that expression \eqref{F g} is equal to 
 \[
 {|I| \choose |I|-2} (u_y+u_{y+1}+\sum_{i\in I} n_i)_{|I|-1} 
 \]
which is invariant under $n_r \leftrightarrow n_{r+1}$. This completes the proof. 
 \end{proof}

\begin{lemma} \label{u independent}Let $l$ be a positive integer, and assume Theorem \ref{t s tau} is true for $N = l$. Then with notation as in the proof of that theorem, the polynomial
\begin{equation} \label{with u}
\sum_{v=l-a}^l  U(v-l+a, \vect{u}) \sum_{\sigma \in S(l,v)} \sum_{\tau_1 \in T(v,v-l+a)} Y(\sigma,\tau_1,\vect{n},\vect{u}) 
\end{equation}
is equal to 
\begin{equation} \label{without u}
\sum_{v=l-a}^l \frac{(l-1)!v}{(l-a)!(l-v)!(v-l+a)!} (X)_{v-l+a}(\sum_{i=1}^l n_i)_{l-v}.
\end{equation}
\end{lemma}

\begin{proof}
We first prove that the polynomial \eqref{with u} is constant in each variable $u_i$. 
Applying identity \eqref{fall identity} to each factor of 
\[
(u_i+\mathrm{degsum}(\sigma_{\tau(i)})_{|\sigma_{\tau(i)}|-1}
\]
in $Y(\sigma,\tau,\vect{n},\vect{u})$, 
 we obtain 
\begin{align}
& \sum_{\tau \in T(v, v-l+a)} \sum_{\sigma \in S(l,v)} Y(\sigma,\tau_1,\vect{n},\vect{u}) \nonumber \\ 
& = \sum_{b=0}^{l-v} \sum_{\vect{c} \in C(l-v-b,v-l+a)}\prod_{i=1}^{v-l+a} (u_i)_{c_i}  \nonumber \\ 
&\sum_{\tau_1 \in T(v, v-l+a)} \sum_{\sigma \in S(l,v)}  \prod_{i=1}^{v-l+a} {|\sigma_{\tau_1(i)}|-1  \choose c_i}( \mathrm{degsum}(\sigma_{\tau_1(i)}))_{|\sigma_{\tau_1(i)}|-1-c_i}  \prod_{i\notin \tau_1,1\leq  i \leq v} ( \mathrm{degsum}(\sigma_i))_{|\sigma_i|-1} \nonumber \\ 
\label{F induction}
\end{align}
where the sum is over the set $C(l-v-b,v-l+a)$ of compositions $\vect{c}$ of $l-v-b$ into $v-l+a$ non-negative parts; that is,  
\[
\vect{c} = (c_1, \ldots, c_{v-l+a}) \text{ and } \sum_{i=1}^{v-l+a} c_i = l-v-b \text{ with } c_i \geq 0.
\]
But the last line of \eqref{F induction} is
\[
F(l-v,\vect{f}, \vect{c})|_{t=1}
\]
where $R$ is the ring of polynomials in $t$, $\delta$ is differentiation with respect to $t$, and $f_i = t^{n_i}, 1 \leq i \leq l$. Evaluating this using the assumption that Theorem \ref{t s tau} is true for $N = l$ thus gives that the right side of equation \eqref{F induction} is equal to 
\begin{align*} 
& \sum_{b=0}^{l-v} \sum_{\vect{c} \in C(l-v-b,v-l+a)} \frac{(l-1)!(l-b)}{(l-a)!b!\prod_{i=1}^{v-l+a}(v-l+a-i+1+\sum_{g=i}^{v-l+a} c_g) }\prod_{i=1}^{v-l+a} {u_i \choose c_i} (\sum_{i=1}^l n_i)_b
\end{align*}

Therefore we must prove that 
\begin{align}
&\frac{(l-1)!}{(l-a)!}\sum_{v=l-a}^l  U(v-l+a, \vect{u}) \nonumber \\
&\sum_{b=0}^{l-v} \frac{(l-b)}{b!}(\sum_{i=1}^l n_i)_b\sum_{\vect{c} \in C(l-v-b,v-l+a)} \frac{1}{\prod_{i=1}^{v-l+a}(v-l+a-i+1+\sum_{g=i}^{v-l+a} c_g) }\prod_{i=1}^{v-l+a} {u_i \choose c_i}  \label{U independence}
\end{align}
is constant in each $u_i$. Note that in the above expression the indeterminate $u_{a-r}$ appears only in terms with $l-r\leq v \leq l$. 

We prove that is constant $u_{a-1}, u_{a-2}, \ldots, u_{a-r}$ for an $1\leq r\leq a-1$. We use induction on $r$ with base case $r=1$. Now in expression \eqref{U independence}, $u_{a-1}$ appears only in terms for $v=l$ and $b=0$ as 
\[
U(a,\vect{u}){l \choose a}
\]
and in $v=l-1,b=0$ as 
\[
U(a-1,\vect{u}){l \choose a}{u_{a-1}\choose 1}
\]
for the composition $\vect{c}$ with $c_i=0, 1 \leq i \leq a-2$ and $c_{a-1}=1$. Then $u_{a-1}$ cancels out in the sum of these two expressions. This proves the base case. 

Now assume that expression \eqref{U independence} is constant in $u_{a-1}, \ldots, u_{a-r}$ for some $1 \leq r \leq a-2$. We thus may assume that $u_i=0$ for $a-r \leq i \leq a-1$, and so only those compositions with $c_i=0$ for $a-r \leq i \leq \mathrm{length}(\vect{c})$ have a non-zero contribution. Fix a $b_0, 0 \leq b_0 \leq a$ and some non-negative integers $\gamma_1, \ldots, \gamma_{a-r-2}$ such that 
\[
r +1- b_0 - \sum_{i=1}^{a-r-2}\gamma_i \geq 0
\]
where we denote the left side of the above inequality by $d$. Now consider the terms in expression \eqref{U independence} with $b=b_0$ and a composition $\vect{c}$ satisfying 
\begin{align*}
c_i &= \gamma_i, \text{ for } 1 \leq i \leq a-r-2 \\ 
c_i &= 0, \text{ for }  a-r\leq i \leq \mathrm{length}(\vect{c}).
\end{align*}
In these terms we must have 
\[
l-v-b_0=\sum_{i=1}^{a-r-2}\gamma_i + c_{a-r-1},
\]
so the quantity $c_{a-r-1}+v$ is fixed. This implies 
\[
v \leq l-b-\sum_{i=1}^{a-r-2}\gamma_i.
\]
And since the $u_{a-r-1}$ appears only in terms with $v\geq l-r-1$, we have that $v$ takes on the $d+1$ values
\[
l-r-1\leq v \leq l-b-\sum_{i=1}^{a-r-2}\gamma_i.
\]
For each of these values of $v$, let $c(v)$ denote the unique composition that satisfies the above requirements. We thus obtain
\begin{align*}
&\sum_{v=l-r-1; c=c(v)}^{l-r-1+d}  \frac{U(v-l+a, \vect{u})}{\prod_{i=1}^{v-l+a}(v-l+a-i+1+\sum_{g=i}^{v-l+a} c_g) } \prod_{i=1}^{v-l+a}{u_{i} \choose c_{i}}\\ 
&=\frac{U(a-r-1,\vect{u})\prod_{i=1}^{a-r-2} {u_i \choose \gamma_{i}}}{\prod_{i=1}^{a-r-2}(a-r-1-i+1+\sum_{g=i}^{a-r-2}\gamma_g)}\frac{1}{d!}\sum_{i=0}^d {d \choose i}(X-(a-r)+1-\sum_{g=1}^{a-r-1}u_g)_i (u_{a-r-1})_{d-i}.
\end{align*}
Now apply identity \eqref{fall identity} to the sum on the right side to see that it is constant in $u_{a-r-1}$. This completes the induction step. 

Now evaluate expression \eqref{U independence} with $u_i=0$. The only non-zero contributions come from the terms with compositions of 0, so $l-v-b=0$. This yields  expression \eqref{without u} and completes the proof.
\end{proof}

\section{Further Work}\label{s further}

\begin{itemize}

\item Study the corresponding Taylor series coefficients when $g(z)$ has more than two terms. 

\item Use the Taylor series to derive the solutions by radicals and see if radical solutions using infinite series can be obtained for higher degree. 

\item Prove a formal factorization of polynomials using these series. 

\item For integer $\beta >1$, interpret these coefficients as counting some kind of tree which generalizes trees with negative vertex degree. 

\item Prove Theorem \ref{t s tau} using induction as used in the proof of Theorem \ref{t deriv set}.

\item See if Theorem \ref{t deriv set} can be generalized using $\tau$-sequences as in Theorem \ref{t s tau} or using more elements $f_A,f_B,f_C\ldots$.

\item Find an NRS-type algorithm that evaluates the other bracket series. 
 
\item See if Theorem \ref{t s tau} can be used to give a proof for any $\beta$.

\item In the case $\beta=1$ find an algebraic proof of formal zeros in terms of trees with negative vertex degree.
 
\end{itemize}

\end{document}